\documentclass{article}
\usepackage{graphicx} 
\usepackage[left=3.5cm,right=3.5cm,top=2.5cm,bottom=2.5cm]{geometry}
\usepackage{amsmath,amssymb,amsthm,mathabx, mathtools, tikz-cd}
\usepackage[colorlinks=true, allcolors=blue]{hyperref}
\usepackage{cleveref}
\usepackage{enumitem}
\renewcommand{\H}{\mathcal H}
\newcommand{\A}{\mathcal A}
\newcommand{\Z}{\mathbb Z}

\DeclareMathOperator{\lab}{\bf{Lab}}
\DeclareMathOperator{\diam}{diam}
\DeclareMathOperator{\Hom}{Hom}

\newcommand{\B}{\mathcal{B}}
\newcommand{\W}{\mathcal{W}}
\newcommand{\V}{\mathcal{V}}
\renewcommand{\S}{\mathcal{S}}

\newcommand{\bg}{\overline{G}}
\newcommand{\bh}{\overline{H}}
\newcommand{\bx}{\pi(X)}
\newcommand{\N}{\mathbb N}

\newcommand{\Q}{\mathbb{Q}}

\newcommand{\eps}{\varepsilon}

\DeclareMathOperator{\tor}{Tor}
\DeclareMathOperator{\coker}{coker}

\DeclareMathOperator{\cd}{cd}
\DeclareMathOperator{\gd}{gd}
\DeclareMathOperator{\hd}{hd}

\DeclareMathOperator{\Out}{Out}
\DeclareMathOperator{\Aut}{Aut}

\newcommand{\X}{\mathfrak{X}}
\newcommand{\F}{\mathfrak{F}}
\newcommand{\I}{\mathfrak{I}}
\newcommand{\HX}[1][\,]{\mathbf{H}_{#1}\mathfrak{X}}
\newcommand{\HF}[1][\,]{\mathbf{H}_{#1}\mathfrak{F}}
\newcommand{\HI}[1][\,]{\mathbf{H}_{#1}\mathfrak{I}}
\newcommand{\minus}{\, \setminus \,}

\renewcommand{\ll }{\langle\hspace{-.7mm}\langle }
\newcommand{\rr }{\rangle\hspace{-.7mm}\rangle }

\newtheorem{thmA}{Theorem}

\newtheorem{corA}[thmA]{Corollary}

\newtheorem{quA}[thmA]{Question}

\newtheorem{theorem}{Theorem}[section]
\theoremstyle{plain}
            \newtheorem{lemma}[theorem]{Lemma}
            \newtheorem{corollary}[theorem]{Corollary}                   
            \newtheorem{proposition}[theorem]{Proposition}  
            
            \newtheorem{claim}{Claim}[theorem]
\theoremstyle{remark}
            \newtheorem{remark}[theorem]{Remark}
            \newtheorem*{remark*}{Remark}
\theoremstyle{definition}
            \newtheorem{definition}[theorem]{Definition}
            \newtheorem{notation}[theorem]{Notation}

\title{Dimensions of finitely generated simple groups \\ and their subgroups}
\author{Francesco Fournier-Facio and Bin Sun}
\date{\today}

\begin{document}

\maketitle

\begin{abstract}
We construct finitely generated simple torsion-free groups with strong homological control.
Our main result is that every subset of $\mathbb{N} \cup \{\infty\}$, with some obvious exceptions, can be realized as the set of dimensions of subgroups of a finitely generated simple torsion-free group. This is new even for basic cases such as $\{ 0, 1, 3 \}$ and $\{ 0, 1, \infty \}$, even without simplicity or finite generation, and answers a question of Talelli and disproves a conjecture of Petrosyan.
Moreover, we prove that every countable group of dimension at least $2$ embeds into a finitely generated simple group of the same dimension. These are the first examples of finitely generated simple groups with dimension other than $2$ or $\infty$.

As another application, we exhibit the first examples of torsion-free groups with the fixed point property for actions on finite-dimensional contractible CW-complexes, and construct torsion-free groups in all countable levels of Kropholler's hierarchy, answering a question of Januszkiewicz, Kropholler and Leary.
Our method combines small cancellation theory with group theoretic Dehn filling, and allows to do several other exotic constructions with control on the dimension.
Along the way we construct the first uncountable family of pairwise non-measure equivalent finitely generated torsion-free groups.
\end{abstract}

\section{Introduction}

Some of the most useful and important invariants of groups are various notions of dimension. These include the \emph{homological} and \emph{cohomological} dimension over a commutative unital ring $R$, and the \emph{geometric} dimension, denoted respectively $\hd_R, \cd_R$ and $\gd$. In this introduction we will talk ambiguously about the \emph{dimension} $\dim$; in our results all the various notions of dimension will coincide.

A very general question that arises in this context is the following.

\begin{quA}
\label{q:spectrum}
    Let $G$ be a group. The \emph{dimension spectrum} of $G$ is the set
    \[S(G) = \{ \dim(H) : H \leq G \} \subseteq \N \cup \{ \infty \}.\]
    Given a set $S \subset \N \cup \{\infty\}$, we say that a group $G$ \emph{realizes} $S$ if $S(G) = S$. We say that $G$ \emph{sharply realizes} $S$ if moreover $\dim(G)$ is only attained by $G$ itself.

    Which sets are (sharply) realized in a given class of groups?
\end{quA}

For example $S(G) \subseteq \{0, \infty \}$ implies that $G$ is torsion, and $S(G) = \{0, 1\}$ is closely related to free groups (the precise statement or conjecture depends on the notion of dimension, see \cite{dunwoody} and \cite[Conjecture II.11]{hd1}).

For most naturally occurring groups, $S(G) = \{ n \in \N \cup \{ \infty \} : n \leq \dim(G) \}$. So a first instance of Question \ref{q:spectrum} is whether \emph{gaps} can occur in the dimension spectrum. Special interest has been paid to \emph{jumps} in cohomology, i.e. spectra satisfying
\[\infty \in S(G) \subseteq \{ n \in \N : n \leq k \} \cup \{ \infty \}\]
for some $k \in \N$. Such jumps are easier to attain for groups with torsion (e.g. for geometric dimension we have $S(D_\infty) = \{ 0, 1, \infty \}$). But for torsion-free groups the situation is more subtle. A question of Talelli \cite[Question 1]{talelli:question}, also recorded in Bestvina's problem list \cite[Question 8.6]{bestvina?questions}, asks whether jumps can occur for cohomological dimension of torsion-free groups. This was generalized by Petrosyan, who conjectured that, for a commutative unital ring $R$, there should be no jump in $\cd_R$ for groups with no $R$-torsion \cite[Conjecture 1.6]{petrosyan:jump1}. A version for Bredon cohomology was then asked by Jo and Nucinkis \cite[Question 5]{jo:nucinkis}. The case of $R = \Q$ and the Bredon version are answered by some branch groups such as Grigorchuk's group \cite{branch1, branch2}, but these examples are full of torsion, moreover the realization is not sharp.

The first instance of our main result answers Talelli's question, and disproves Petrosyan's conjecture for all commutative unital rings:

\begin{thmA}[Theorem \ref{thm. infdim tarski}]
\label{intro:thm:jump}
    There exist continuum many pairwise non-isomorphic finitely generated torsion-free groups, each of which sharply realizes $\{ 0, 1, \infty \}$.
\end{thmA}

Jumps as in Theorem \ref{intro:thm:jump} do not occur for groups in Kropholler's hierarchy $\HF$ \cite[Theorem 3.2]{petrosyan:jump1}. Therefore Theorem \ref{intro:thm:jump} gives new examples of torsion-free groups that do not lie in $\HF$. In fact, something much stronger is true:

\begin{corA}[Corollary \ref{cor:smith}]
\label{intro:cor:smith}
    There exist continuum many pairwise non-isomorphic finitely generated torsion-free groups, such that every action on a finite-dimensional contractible CW-complex has a global fixed point.
\end{corA}

Such groups are sometimes called \emph{Smith groups}, in reference to Smith's Theorem that finite $p$-groups have this property \cite{smith:theory}. The first examples of infinite Smith groups were constructed in \cite{infinite:fixpoint}. The existence of torsion-free examples was an open problem communicated to us by Kropholler and Leary.

\medskip

We distinguish the groups in Theorem \ref{intro:thm:jump} up to isomorphism by estimating their $L^2$-Betti numbers, so we also obtain:

\begin{thmA}[Theorem \ref{thm. infdim tarski}]
\label{intro:thm:ME}
    There exist continuum many pairwise non-measure equivalent finitely generated torsion-free groups.
\end{thmA}

To the best of our knowledge, this is the first such construction. Measurably diverse finitely generated groups have been constructed independently by Ioana--Tucker-Drobb \cite{ioana:ME}; however their groups contain torsion. Kac--Moody groups over finite fields are a source of measurably diverse finitely presented simple groups \cite{lopezneumann}, although these also contain torsion, and of course cannot lead to an uncountable family.

\medskip

The question of gaps also arises naturally in finite-dimensional groups. In this respect, it is particularly interesting to determine whether for a group $G$ there always exist subgroups of \emph{codimension one}, intended in a suitable geometric sense. For example, finding codimension one subgroups in the fundamental groups of hyperbolic $3$-manifolds \cite{kahn:markovic} is a fundamental step towards their cubulation \cite{bergeron2012boundary} which in turn is central in the proof of the Virtual Haken Conjecture \cite{wise2021structure, agol2013virtual}. Moreover, constructing enough codimension one subgroups in general $3$-dimensional hyperbolic groups provides a possible approach to the Cannon Conjecture \cite{markovic}.
A result in this direction is that there exist $n$-dimensional hyperbolic groups that do not contain fundamental groups of closed orientable hyperbolic $k$-manifolds, for any $k > 2$ \cite{systolic1, systolic2}.

In the context of Question \ref{q:spectrum}, one can consider subgroups $H \leq G$ such that $\dim(H) = \dim(G) - 1$, which gives a homological notion of codimension one. The existence of such gaps was the subject of recent questions \cite{gapquestion1, gapquestion2}. See also \cite[Question 5.1]{wilton2024surface}, which is strongly related to the problem of sharp realization in cubulable hyperbolic groups. In the pro-$p$ setting an answer is more readily available. For $p$-adic analytic groups, the virtual cohomological dimension coincides with the $p$-adic dimension \cite[Theorem 5.1.9]{padic}; since the Lie algebra $\mathfrak{sl}_3(\mathbb{Q}_p)$ has no subalgebra of codimension one, this implies that the first congruence subgroup of $\mathrm{SL}_3(\mathbb{Z}_p)$ has no subgroup of (virtual) cohomological codimension one. However, it is not clear whether this observation can be used to obtain, say, an arithmetic group with a gap. As a disclaimer, there is another, more metric notion of dimension spectrum for profinite groups, which happens to coincide with the virtual cohomological one in the case of $p$-adic analytic pro-$p$ groups, up to a normalization \cite[Theorem 1.1]{barnea:shalev}.

The second instance of our main result shows that arbitrarily large finite gaps can occur in general.

\begin{thmA}[Corollary \ref{cor. 3 tarski} and Theorem \ref{thm:fd:continuum}]
\label{intro:thm:fd:gap}
For all $n \geq 3$, there exist infinitely many pairwise non-isomorphic finitely generated (torsion-free) groups, each of which sharply realizes $\{ 0, 1, n\}$.
Moreover, if $n \geq 4$, this can be strengthened to continuum many.
\end{thmA}

The different behavior of $\{0, 1, 3\}$ is an artifact of the proof, and can be removed if there exist $3$-dimensional hyperbolic groups with prescribed torsion in $H^3$ (Remark \ref{rem:3dim}).

In light of the previous discussion, it is particularly interesting to determine whether such gaps can occur in hyperbolic groups, especially of dimension $3$. In this context, let us point out that the examples from Theorem \ref{intro:thm:fd:gap} can be easily arranged to be \emph{lacunary} hyperbolic, see \cite[Theorem 4.26]{olshanskii2009lacunary} and its proof.

More generally, we do not know if groups as in Theorem \ref{intro:thm:fd:gap} can ever be finitely presented. By Serre's Theorem \cite{serre}, a torsion-free group of cohomological dimension $n \geq 3$ answering \cite[Question 0.1]{wilton2024surface} would provide an example. Our groups are not finitely presentable, essentially by construction (Remark \ref{rem:fp}).

\medskip

Both of these results are just very specific examples of what sets can be realized. There are a few obvious exceptions: the dimension spectrum of a non-trivial torsion-free group $G$ must contain $\{0, 1\}$, where equality holds (essentially) only for free groups;
moreover if $S(G)$ is infinite then $\dim(G) = \infty$ (Lemma \ref{lem. jump exceptions}). It turns out that \emph{everything else is possible}, and moreover this can be achieved with \emph{finitely generated torsion-free simple groups}. This is our main result:

\begin{thmA}[Theorem \ref{thm. realization}]
\label{intro:thm:spectrum}
Let $\{0, 1\} \subsetneq S \subseteq \N \cup \{ \infty \}$. Assume that either $S$ is finite or $\infty \in S$. Then there exists a finitely generated simple torsion-free group $G_S$ which sharply realizes $S$.
\end{thmA}

The study of finitely generated infinite simple groups is a central topic in group theory. The first example was constructed by Higman \cite{higmangroup}, and it uses Zorn's Lemma, so it is very implicit. An explicit uncountable family was constructed by Camm \cite{camm}; these groups arise as amalgamated products of free groups and therefore they have dimension $2$. Thompson introduced two finitely presented infinite simple groups $T$ and $V$ in unpublished notes \cite{cannonfloydparry}. This construction was generalized in many different directions, including some torsion-free groups \cite{hydelodha2019, flows, hydelodha2023}. All of these groups contain infinite direct sums and so are infinite-dimensional. There are two more sources of finitely presented infinite simple groups: lattices in products of trees \cite{burgermozes}, which are intrinsically $2$-dimensional, and Kac--Moody groups over finite fields \cite{capraceremy}, which have torsion. Other constructions of finitely generated simple groups include the ones obtained via small cancellation theory, which a priori do not give information on the dimension, unless they come with an explicit aspherical presentation \cite{obraztsov1990imbedding}, in which case they are $2$-dimensional.

In all of these examples the dimension (when known) is either $2$ or infinite. So the groups in Theorem \ref{intro:thm:spectrum} are the first examples of finitely generated simple groups whose dimension is high but finite.

\begin{corA}[Corollary \ref{cor:simple}]
\label{intro:cor:simple}
    For all $n \geq 3$, there exist continuum many pairwise non-isomorphic finitely generated simple (torsion-free) groups $G$ with $\dim(G) = n$.
\end{corA}

In fact we can prove an embedding theorem into finitely generated simple groups preserving the dimension. It is well-known that every countable group embeds into a finitely generated simple group. There are several proofs of this result. The original constructions are due to Hall \cite{embeddingsimple:hall}, which creates torsion, and Gorju\v{s}kin \cite{embeddingsimple:gorjuskin}, which implicitly uses Zorn's Lemma. Other constructions are due to Schupp \cite{embeddingsimple:hall} and Thompson \cite{embeddingsimple:thompson}, which create torsion. The latter construction was recently refined within the context of left orderable groups \cite{darbinyansteenbock}, so the resulting groups are torsion-free, but still contain infinite direct sums.
Also here, other constructions include ones obtained via small cancellation theory, which a priori do not give information on the dimension.

\begin{thmA}[Theorem \ref{thm:simple:embedding}]
\label{intro:thm:simple:embedding}
    Let $n \geq 2$ and let $H$ be a countable group with $\gd(H) \leq n$. Then $H$ embeds into a finitely generated simple group $G$ with $\dim(G) = n$.
\end{thmA}

The groups we construct for Theorems \ref{intro:thm:jump} and \ref{intro:thm:fd:gap} are \emph{torsion-free Tarski monsters}: finitely generated infinite simple groups with every proper non-trivial subgroup infinite cyclic. Such groups were first constructed by Ol'{\v{s}}anski\u{\i} \cite{noetheriangroup}; his construction gives an explicit aspherical presentation \cite{ol2012geometry}, and therefore they are $2$-dimensional.

The novelty in our approach is that the small cancellation relations used to construct such exotic groups are arranged to also have the \emph{Cohen--Lyndon property}. This property, named after \cite{cohen1963free}, ensures that the normal closure of a relator takes a particularly nice form, and appears naturally in group theoretic Dehn filling \cite{sun2018cohomologyi}. It gives homological control \cite{sun2019cohomologyii, petrosyan2024l2}, which is robust enough to allow us to understand homological features of the limit. This property has already been applied to produce high-dimensional versions of small cancellation constructions \cite{sun2019cohomologyii, Arenas}. The relations needed to construct groups such as torsion-free Tarski monsters \cite{olshanskii1993residualing, osin2010small, hull2013small} all impose that a given element belongs to a given non-elementary subgroup. Such relations can be easily chosen to have the Cohen--Lyndon property (via group theoretic Dehn fillings) if torsion is allowed (see e.g. \cite[Theorem 1.1]{osinthom}); but in torsion-free groups this requires more work. We achieve this in our main technical result: Theorem \ref{thm. inductive step}.

\medskip

In order to showcase the flexibility of our method, let us present three further applications. We are able to construct groups as in \cite{olshanskii1993residualing, osin2010small, hull2013small, characteristicquotients}, with control on the dimension.

\begin{thmA}[Theorem \ref{thm:verballycomplete}]
\label{intro:thm:verballycomplete}
    For all $n \geq 2$, there exists a finitely generated verbally complete group $G$ with $\dim(G) = n$.
\end{thmA}

\begin{thmA}[Theorem \ref{thm:conjugacyclass}]
\label{intro:thm:conjugacyclass}
    For all $n \geq 2$, there exists a finitely generated group $G$ such that $G$ has exactly two conjugacy classes and $\dim(G) = n$.
\end{thmA}

\begin{thmA}[Theorem \ref{thm:characteristicquotients}]
\label{intro:thm:characteristicquotients}

    For all $n \geq 2$, and all $d \geq 2(n-1)$, the free group $F_n$ admits an infinite simple characteristic quotient $G$ with $\cd_{\Q}(G) = d$.
\end{thmA}

For the full extent of Theorem \ref{intro:thm:spectrum}, and for Theorem \ref{intro:thm:simple:embedding}, we use a \emph{relative} version of the Tarski monster construction. Without the homological control, these were already constructed by Obraztsov \cite{obraztsov1990imbedding} and Ol'{\v{s}}anski\u{\i} \cite{olshanskiieconomical}. However this formulation is new; we believe it could be useful in other contexts, so we isolate the statement as such, with more details given in the text.

\begin{thmA}[Theorem \ref{thm. embedding}]
\label{intro:thm:relative}

    Every countable torsion-free group $H$ embeds as a malnormal subgroup in a finitely generated simple torsion-free group $G$ such that every proper subgroup of $G$ is either cyclic or conjugate into $H$, and the map $H \to G$ induces isomorphisms
    \[H_n(G;A) \cong H_n(H;A), \; \; H^n(G;A) \cong H^n(H;A)\]
    for all $n\geq 3$ and every $G$-module $A$. 
\end{thmA}

In the results presented so far, the main focus was the dimension. But Theorem \ref{intro:thm:relative} (and its more precise version Theorem \ref{thm. embedding}) gives more control, with applications of a different flavor.

Let $\X$ be a class of groups. We define $\HX$ to be the smallest class of groups containing $\X$, and with the property that if $G$ is a group with an admissible action on a finite-dimensional contractible CW-complex with stabilizers in $\HX$, then $G \in \HX$. This hierarchy was introduced by Kropholler \cite{kropholler:hierarchy, kropholler:hierarchy:survey}, who proved several rigidity theorems for groups in $\HF$, where $\F$ is the class of finite groups. Most notably, if $G$ is a torsion-free group of type $\mathrm{FP}_\infty$ in $\HF$, then $G$ has finite cohomological dimension \cite{kropholler:hierarchy}; this was then generalized for groups with torsion, and classifying spaces for proper actions \cite{krophollermislin}.

This class admits a filtration by ordinals $\HX = \bigcup_\alpha \HX[\alpha]$: see Subsection \ref{ss. hierarchy}. Most of the proofs proceed by transfinite induction, and therefore it is central to understand the precise role of the ordinal $\alpha$. This led Januszkiewicz, Kropholler and Leary to construct groups in $\HF \minus \HF[\alpha]$ for all countable ordinals $\alpha$ \cite{JKL}. Their construction takes as input groups with strong fixed point properties for actions on finite dimensional complexes \cite{infinite:fixpoint}. These fixed point properties come from torsion, therefore their method can only be used to address $\HX$ when $\X$ is a class that contains all finite groups. In particular, it was an open problem to construct torsion-free groups in $\HF \minus \HF[\alpha]$, beyond the case $\alpha = 2$ (these examples for $\alpha = 2$ are constructed in \cite{HF:torsionfree1, HF:torsionfree2} and discussed further in \cite[Section 2]{JKL}).

\begin{thmA}[Theorem \ref{thm. hierarchy}]
\label{intro:thm:hierarchy}
    Let $\X$ be a subgroup-closed class of groups. Suppose that there exists a countable torsion-free group in $\HX[1] \minus \X$. Then, for every countable ordinal $\alpha\geq 1$, there exists a finitely generated torsion-free simple group in $\HX[\alpha+1] \minus \HX[\alpha]$.
\end{thmA}

It is easy to see that $\HX[1]$ contains all groups of finite cohomological dimension, which are automatically torsion-free, so the assumption holds as soon as $\X$ does not contain \emph{all} groups of finite cohomological dimension. In the case of $\X = \F$, we obtain torsion-free groups in $\HF \minus \HF[\alpha]$, for every countable ordinal $\alpha$, solving a problem posed in \cite[after Theorem 1.1]{JKL}. The importance of having torsion-free examples comes from the fact that the torsion-free groups in $\HF$ are exactly the groups in $\HI[\alpha]$, where $\I$ is the class consisting only of the trivial group; this is an easy consequence of Smith theory \cite{smith:theory}.

\paragraph{Outline.} In Section \ref{sec:prelim} we lay the necessary foundations from geometric group theory, particularly small cancellation theory over acylindrically hyperbolic groups and the Cohen--Lyndon Property. In Section \ref{sec:inductive} we prove a small cancellation theorem with homological control (Theorem \ref{thm. inductive step}), which will give the inductive step in our constructions. In Section \ref{sec. relative} we construct relative torsion-free Tarski monsters with homological control, which proves Theorem \ref{intro:thm:relative} and is used in all the subsequent applications, treated in Section \ref{sec:applications}.

Except for Subsection \ref{ss. further}, all applications only rely on Theorem \ref{thm. embedding}, so the reader that does not want to delve into the small cancellation theory can skip Sections \ref{sec:inductive} and \ref{sec. relative} and treat Theorem \ref{thm. embedding} as a black box.

\paragraph{Acknowledgements.} The authors are indebted to Henrique Souza for asking the question that motivated this paper, and is solved in Theorem \ref{intro:thm:fd:gap}; and to Peter Kropholler for suggesting that our techniques could be used to also prove Theorem \ref{intro:thm:hierarchy}. They thank Macarena Arenas, Rachael Boyd, Will Cohen, Arman Darbinyan, Daniel Groves, Adrian Ioana, Peter Kropholler, Ian Leary, Kevin Li, Clara L{\"o}h, Antonio L{\'o}pez Neumann, Ashot Minasyan, Brita Nucinkis, Nansen Petrosyan, John Ratcliffe, Yuri Santos Rego, Olympia Talelli, Henry Wilton, Julian Wykowski and Matt Zaremsky for useful conversations. FFF is supported by the Herchel Smith Postdoctoral Fellowship Fund. BS is supported by the AMS--Simons Travel Grant.

\section{Preliminaries}
\label{sec:prelim}

\subsection{(Co)homology and dimension}

We refer the reader to \cite{brown1982cohomology} or \cite{bieri1981homdim} for more details. Let $G$ be a group and $R$ a commutative unital ring. The \emph{$R$-homological dimension of $G$} is defined as
\[\hd_R(G):=\sup_n \{ n \in \N : H_n(G;M)\neq 0 \text{ for some $RG$-module $M$} \} \in \N \cup \{ \infty \}.\]
The \emph{$R$-cohomological dimension of $G$}, denoted $\cd_R(G)$, is defined similarly.
The \emph{geometric dimension of $G$}, denoted $\gd(G)$, is the minimal dimension of a $K(G,1)$ CW-complex, and equals $\infty$ if no finite-dimensional $K(G,1)$ exists. For a commutative unital ring $R$, one always has
\[\hd_R(G)\leq \cd_R(G)\leq \gd(G).\]
In particular:
\begin{itemize}
    \item If $G$ has an $n$-dimensional $K(G, 1)$, and $H_n(G; R) \neq 0$ for every commutative unital ring $R$, then all dimensions above are equal to $n$.
    \item If $H_n(G; R) \neq 0$ for infinitely many $n$ for every commutative unital ring $R$, then all dimensions above are equal to infinity.
\end{itemize}
This is how we will compute the dimension of the groups in our construction, which is why the results in the introduction hold for all notions of dimension at once (with the exception of Theorem \ref{intro:thm:characteristicquotients}).

\begin{lemma}\label{lem. compute cohomology}
    Let $G$ be a group and let $H$ be a quotient of $G$. Let also $X$ (resp. $Y$) be a $K(G,1)$ (resp. $K(H,1)$) CW-complex. Suppose that $Y$ is obtained from $X$ by attaching (possibly infinitely many) $1$-cells and $2$-cells. Then, for every $H$-module $A$, the quotient $G \to H$ induces:
    \begin{itemize}
        \item Isomorphisms $H_j(G;A)\cong H_j(H;A)$ and $H^j(H; A) \cong H^j(G; A)$, for all $j \geq 3$;
        \item An embedding $H_2(G;A) \hookrightarrow H_2(H;A)$;
        \item A surjection $H^2(H; A) \twoheadrightarrow H^2(G;A)$.
    \end{itemize}
\end{lemma}

\begin{proof}
    Let $\widetilde X$ (resp. $\widetilde Y$) be the universal cover of $X$ (resp. $Y$) and let $C_\ast(\widetilde X)$ (resp. $C_\ast(\widetilde Y)$) be the cellular chain complex of $\widetilde X$ (resp. $\widetilde Y$). As $Y$ is obtained from $X$ by attaching $1$-cells and $2$-cells, the $H$-modules $C_j(\widetilde X)\otimes_{\Z G}\Z H$ and $C_j(\widetilde Y)$ are the same for all $j \geq 3$. Moreover, $C_2(\widetilde X)\otimes_{\Z G}\Z H$ is a direct summand of $C_2(\widetilde Y)$, seen as $H$-modules.
    
    Let $d_\ast\colon C_\ast(\widetilde Y)\to C_{\ast-1}(\widetilde Y)$ be the boundary map of $C_\ast(\widetilde Y)$. Then for $j\geq 3$, $d_j$ is also the boundary map from $C_j(\widetilde X)\otimes_{\Z G}\Z H$ to $C_{j-1}(\widetilde X)\otimes_{\Z G}\Z H$. Now for every $H$-module $A$, we also have identifications
    \[C_\ast(\widetilde X)\otimes_{\Z G}A=C_\ast(\widetilde X)\otimes_{\Z G}\Z H\otimes_{\Z H}A\] 
    and 
    \[\Hom_{\Z G}(C_\ast(\widetilde X),A)\cong \Hom_{\Z H}(C_\ast(\widetilde X)\otimes_{\Z G}\Z H,A).\]
    These induce the following commutative diagrams:
\begin{center}
    \begin{tikzcd}
       \dots \arrow[r] &C_4(\widetilde X)\otimes_{\Z G}A \arrow[d, "\cong"] \arrow[r] &C_3(\widetilde X)\otimes_{\Z G}A \arrow[d, "\cong"] \arrow[r] &C_2(\widetilde X)\otimes_{\Z G}A \arrow[d, hook]\\
       \dots \arrow[r] &C_4(\widetilde Y)\otimes_{\Z H}A \arrow[r] &C_3(\widetilde Y)\otimes_{\Z H}A \arrow[r] &C_2(\widetilde Y)\otimes_{\Z H}A
    \end{tikzcd}
\end{center}

\begin{center}
    \begin{tikzcd}
       \dots  & \Hom_{\Z G}(C_4(\widetilde X),A)\arrow[l] &\Hom_{\Z G}(C_3(\widetilde X),A) \arrow[l] &\Hom_{\Z G}(C_2(\widetilde X),A)\arrow[l]\\
       \dots &\Hom_{\Z H}(C_4(\widetilde Y),A)\arrow[u,swap,"\cong"] \arrow[l] &\Hom_{\Z H}(C_3(\widetilde Y),A)\arrow[u,swap,"\cong"] \arrow[l] &\Hom_{\Z H}(C_2(\widetilde Y),A)\arrow[u,two heads] \arrow[l]
    \end{tikzcd}
\end{center}

It follows immediately that the induced maps in degrees $j \geq 4$ are isomorphisms. Since $C_3(\widetilde Y) = C_3(\widetilde X)\otimes_{\Z G}\Z H$, the image of the boundary map $C_3(\widetilde Y) \to C_2(\widetilde Y)$ takes values in $C_2(\widetilde X)\otimes_{\Z G} \Z H$. This gives the isomorphism in degree $3$.

For homology in degree $2$, the injectivity follows from diagram chasing. A class in $H_2(G; A)$ vanishes in $H_2(H; A)$ if a cycle $z \in C_2(\widetilde X) \otimes_{\Z G} A$ representing it has a preimage in $b \in C_3(\widetilde Y) \otimes_{\Z H} A$. We can consider the corresponding element $b' \in C_3(\widetilde X) \otimes_{\Z G} A$, and this will be a preimage of $z$, witnessing that $z$ represents the trivial class in $H_2(G; A)$.

For cohomology in degree $2$, since $C_2(\widetilde X)$ is  a direct summand of $C_2(\widetilde Y)$, the surjection $\Hom_{\Z H}(C_2(\widetilde Y),A) \twoheadrightarrow \Hom_{\Z G}(C_2(\widetilde X),A)$ is a retraction, and surjectivity follows.
\end{proof}

\subsection{$L^2$-Betti numbers}

The non-measure equivalence in Theorem \ref{intro:thm:jump} will be established via $L^2$-Betti numbers. We recall the basic definitions here and refer to \cite{luck2002l2} for details. Let $G$ be a discrete group. The \emph{group von Neumann algebra} of $G$, denoted $\mathcal{N}(G)$, is the algebra of all bounded linear operators on $\ell^2(G)$ that commute with the left regular representation of $G$. The group ring $\Z G$ is naturally a subring of $\mathcal{N}(G)$, which endows $\mathcal{N}(G)$ with a left $\Z G$-module structure.

Let $H$ be a quotient of $G$. Note that the homology $H_\ast(G; \mathcal{N}(H))$ is naturally a right $\mathcal{N}(H)$-module. Indeed, fix a contractible CW-complex $X$ with a free right $G$-action and let $C_\ast(X)$ be the cellular chain complex of $X$, which is a chain complex consisting of right $\Z G$-modules. The homology $H_\ast(G;\mathcal{N}(H))$ is the homology of the chain complex $C_\ast(X)\otimes_{\Z G}\mathcal{N}(H)$, which is thus a right $\mathcal{N}(H)$-module. 

Associated to each $\mathcal{N}(H)$-module $M$ is its \emph{von Neumann dimension}, denoted $\dim_{\mathcal{N}(H)} M$. The reader is referred to \cite[Section 6.3]{luck2002l2} for a definition. We will denote
\[b^{(2)}_n(G;H):=\dim_{\mathcal{N}(H)} H_n(G;\mathcal{N}(H)),\;\; b^{(2)}_n(G):=b^{(2)}_n(G;G),\;\; n\in \N.\]
The latter is called the \emph{$n^{\text{th}}$ $L^2$-Betti number} of $G$.

The following result allows us to use $L^2$-Betti numbers to distinguish groups up to measure equivalence - in particular we will not need the definition of measure equivalence anywhere in the paper.

\begin{theorem}[{\cite[Theorem 6.3]{gaboriau2002invariants}}]
\label{thm. gaboriau}
    Let $G,H$ be countable discrete groups. Suppose that they are measure equivalent with measure equivalence index $c$. Then $b^{(2)}_\ast(G)=c \cdot b^{(2)}_\ast(H)$.
\end{theorem}

That is, the $L^2$-Betti numbers of measure equivalent groups are proportional. What is crucial for us, is that the proportionality constant is independent of the degree.

\begin{corollary}[{\cite[Proposition 2.6]{cheeger1986l2}}]
\label{cor. fi l2 betti}
    Let $G$ be a countable group with a finite-index subgroup $H$. Then $b^{(2)}_\ast(H)=[G:H] \cdot b^{(2)}_\ast(G)$.
\end{corollary}

We will also need an approximation result. The following is a special case of \cite[Corollary 3.4]{petrosyan2024l2}, which is an easy consequence of \cite[Theorem 1.5]{jaikin2020strong}.

\begin{theorem}\label{thm. l2 betti approximation}
    Let $G$ be a type $F$ virtually locally indicable group. Then for every $k,\delta>0$, there exists a finite subset $\mathcal{F}_{k,\delta}\subset G\smallsetminus\{1\}$ such that if a normal subgroup $N\lhd G$ satisfies $N\cap \mathcal{F}_{k, \delta}=\emptyset$, then for all $n\leq k$, we have
    \[|b^{(2)}_n(G)-b^{(2)}_n(G ; G/N)|<\delta.\]
\end{theorem}

\subsection{Small cancellation theory}

Throughout this section, we use $G$ to denote a group, $\{H_\lambda\}_{\lambda\in\Lambda}$ a family of subgroups of $G$, $X=X^{-1}$ a symmetric subset of $G$. If $G$ is generated by $X$ together with the union of all $H_{\lambda}$ we say that $X$ is a \emph{relative generating set} of $G$ with respect to $\{H_{\lambda}\}_{\lambda\in\Lambda}$. In this case we denote: 

\begin{equation}\label{eq. alphabet 0}
    \H=\bigsqcup_{\lambda\in\Lambda}H_{\lambda},\;\; \A=X\sqcup\H.
\end{equation}

We use $\A^\ast$ to denote words in the alphabet $\A$. Note that this alphabet is symmetric.

\subsubsection{Hyperbolically embedded subgroups}\label{subsec. he}

The notion of hyperbolically embedded subgroups was introduced by Dahmani--Guirardel--Osin \cite{dahmani2017hyperbolically}. 
Consider the Cayley graph $\Gamma(G,\A)$. Note that, for $\lambda\in\Lambda$ there is a natural embedding $\Gamma(H_{\lambda},H_{\lambda})\hookrightarrow \Gamma(G,\A)$ whose image is the subgraph of $\Gamma(G,\A)$ with vertices and edges labeled by elements of $H_{\lambda}$.

\begin{remark}
We allow $X\cap H_{\lambda}\neq \emptyset$ and $H_{\lambda}\cap H_{\mu}\neq\{1\}$ for distinct $\lambda,\mu\in\Lambda$, in which case there will be multiple edges between some pairs of vertices of $\Gamma(G,\A)$.
\end{remark}

For $\lambda\in\Lambda$, an edge path in $\Gamma(G,\A)$ between vertices of $\Gamma(H_{\lambda},H_{\lambda})$ is called $H_{\lambda}$\emph{-admissible} if it does not contain any edge of $\Gamma(H_{\lambda},H_{\lambda})$. Note that an $H_{\lambda}$-admissible path is allowed to pass through vertices of $\Gamma(H_{\lambda},H_{\lambda})$. 

\begin{definition}\label{def. relative metric}
For every pair of elements $h,k\in H_{\lambda}$, let $\widehat{d}_{\lambda}(h,k)\in[0,\infty]$ be the length of a shortest $H_{\lambda}$-admissible path connecting the vertices labeled by $h$ and $k$. If no such path exists, set $\widehat{d}_{\lambda}(h,k)=\infty$. The laws of summation on $[0,\infty)$ extend naturally to $[0,\infty]$ and it is easy to verify that $\widehat{d}_{\lambda}:H_{\lambda}\times H_{\lambda}\rightarrow [0,+\infty]$ defines a metric on $H_{\lambda}$, which is called the \emph{relative metric on} $H_{\lambda}$ \emph{with respect to} $X$.
\end{definition}

\begin{definition}\label{hyperbolically embedded}
We say that the family $\{H_\lambda\}_{\lambda\in\Lambda}$ \emph{hyperbolically embeds into} $(G,X)$ (denoted by $\{H_\lambda\}_{\lambda\in\Lambda}\hookrightarrow_h(G,X)$) if the following hold:
\begin{itemize}
    \item $G$ is generated by $\A = X \sqcup \H$;
    \item the Cayley graph $\Gamma(G,X\sqcup \mathcal{H})$ is a Gromov hyperbolic space;
    \item for each $\lambda\in\Lambda$, the metric space $(H_{\lambda},\widehat{d}_{\lambda})$ is proper, i.e., every ball of finite radius contains only finitely many elements.
\end{itemize}
If in addition, $X$ and $\Lambda$ are finite, then we say that $G$ is \emph{hyperbolic relative to $\{H_\lambda\}_{\lambda\in\Lambda}$} or $(G,\{H_\lambda\}_{\lambda\in\Lambda})$ is a \emph{relatively hyperbolic pair}. 
Further, we say that the family $\{H_\lambda\}_{\lambda\in\Lambda}$ \emph{hyperbolically embeds into} $G$, denoted by $\{H_\lambda\}_{\lambda\in\Lambda}\hookrightarrow_h G$, if there exists some subset $X\subset G$ such that $\{H_\lambda\}_{\lambda\in\Lambda}\hookrightarrow_h(G,X)$.
\end{definition}

The reader is referred to \cite[Proposition 4.28]{dahmani2017hyperbolically} for the equivalence between the above definition of relative hyperbolicity and one of the standard definitions.

\begin{notation}
In case $\{H_\lambda\}_{\lambda\in\Lambda}=\{H\}$ is a singleton, we will drop the braces and write $H\hookrightarrow_h G$.
\end{notation}

The next lemma will be useful to change generating sets in the proof of Theorem \ref{thm. inductive step}.

\begin{lemma}[{\cite[Corollary 4.27]{dahmani2017hyperbolically}}]\label{lem. change relative generating set}
    Let $G$ be a group, $\{H_\lambda\}_{\lambda\in \Lambda}$ a family of subgroups of $G$, and $X_1,X_2\subset G$ relative generating sets of $G$ with respect to $\{H_\lambda\}_{\lambda\in \Lambda}$. Suppose that the symmetric difference of $X_1$ and $X_2$ is finite. Then $\{H_\lambda\}_{\lambda\in \Lambda}\hookrightarrow_h(G,X_1)$ if and only if $\{H_\lambda\}_{\lambda\in \Lambda}\hookrightarrow_h(G,X_2)$.
\end{lemma}

An important property of hyperbolically embedded subgroups, which will be useful in proving simplicity statements, is the following.

\begin{definition}
    Let $G$ be a group and $\{ H_i \}_{i \in I}$ a family of subgroups. We say that $\{ H_i \}_{i \in I}$ is \emph{almost malnormal} if $|gH_i g^{-1} \cap H_j| = \infty$, for some $g\in G$ and $i, j \in I$, implies $i = j$ and $g \in H_i$. It is \emph{malnormal} if the same conclusion is reached by only assuming $gH_i g^{-1} \cap H_j \neq \{ 1 \}$.
\end{definition}

\begin{lemma}[{\cite[Proposition 4.33]{dahmani2017hyperbolically}}]\label{lem. he -> malnormal}
    Let $G$ be a group with a hyperbolically embedded family of subgroups $\{H_\lambda\}_{\lambda\in \Lambda}$. Then $\{H_\lambda\}_{\lambda\in \Lambda}$ is an almost malnormal family of $G$. In particular, if $G$ is moreover torsion-free, then $\{H_\lambda\}_{\lambda\in \Lambda}$ is a malnormal family of $G$.
\end{lemma}

The next theorem gives a necessary and sufficient condition for enlarging a hyperbolically embedded family by keeping the same relative generating set.

\begin{theorem}[{\cite[Theorem 3.9]{antolin2016commensurating}}]\label{thm. criterion for he}
Suppose that $G$ is a group, $\{H_\lambda\}_{\lambda\in\Lambda}$ is a family of subgroups of $G$, and $X\subset G$ is a subset such that $\{H_\lambda\}_{\lambda\in\Lambda}\hookrightarrow_h(G,X)$. Set $\A=X\sqcup(\bigsqcup_{\lambda\in \Lambda}H_\lambda)$. A family of subgroups $\{Q_i\}^n_{i=1}$ satisfies the three conditions below, if and only if $\{Q_i\}^n_{i=1}\cup\{H_\lambda\}_{\lambda\in\Lambda}\hookrightarrow_h(G,X)$.
\begin{enumerate}
    \item[(C1)] Each $Q_i$ is generated by a finite subset $Y_i\subset Q_i$ as a group.
    \item[(C2)] There exist $\mu\geq 1$ and $c\geq 0$ such that for all $i$ and all $h\in Q_i$, we have $|h|_{Y_i}\leq \mu|h|_{\A}+c$, where $|\cdot|_{Y_i}$ denotes the word length of $h$ with respect to the generating set $Y_i$, and $|\cdot|_{\A}$ is defined similarly.
    \item[(C3)] For every $\eps>0$, there exists $R>0$ such that the following holds. Suppose that for some $g\in G$ and $i,j\in\{1,\dots,n\}$, we have
    $$\diam(Q_i\cap (gQ_j)^{+\eps})\geq R,$$
    then $i=j$ and $g\in Q_i$, where $(gQ_j)^{+\eps}$ denotes the $\eps$-neighborhood of $gQ_j$ in $\Gamma(G,\A)$.
\end{enumerate}
\end{theorem}

\subsubsection{Acylindrical hyperbolicity}

The notion of an acylindrically hyperbolic group was introduced by Osin \cite{osin2016acylindrically}. This is based on the notion of an acylindrical action, which was introduced by Bowditch \cite{bowditch2008tight} but the idea dates back to Sela \cite{sela1997acylindrical}.

\begin{definition}
An action of $G$ on a metric space $X$ is \emph{acylindrical} if for all $\eps>0$ there exists $R>0$ and $n \in \N$ such that: for all $x, y \in X$ with $d(x, y) \geq R$, the set $\{g \in G : d(x, gx), d(y, gy) \leq \eps \}$ has cardinality at most $n$.

A group is \emph{acylindrically hyperbolic} if it admits a non-elementary acylindrical action on a hyperbolic space.
\end{definition}

This can be equivalently characterized in terms of hyperbolically embedded subgroups: namely $G$ is acylindrically hyperbolic if and only if there exists a proper infinite subgroup $H$ and a hyperbolic embedding $\{ H \} \hookrightarrow_h G$ \cite[Theorem 1.2]{osin2016acylindrically}. The following combines Proposition 5.2 and Theorem 5.4 of \cite{osin2016acylindrically}.

\begin{theorem}[{\cite{osin2016acylindrically}}]\label{thm. acylindrical action}
    If $\{ H_\lambda \}_{\lambda \in \Lambda} \hookrightarrow_h (G,Y)$ for some $Y\subset G$, then there exists $Y\subset X \subset G$ such that $\{ H_\lambda \}_{\lambda \in \Lambda} \hookrightarrow_h (G, X)$ and the action of $G$ on $\Gamma(G, X \sqcup \H)$ is acylindrical. If $|Y|<\infty$, one can let $X=Y$.
\end{theorem}

Every acylindrically hyperbolic group $G$ admits a \emph{finite radical}, denoted $K(G)$, i.e. a unique maximal finite normal subgroup \cite[Theorem 3.23]{dahmani2017hyperbolically}. Finite normal subgroups can be problematic for small cancellation theory, hence the following definition:

\begin{definition}
    Let $G$ be a group with an acylindrical action on a Cayley graph $\Gamma(G, \A)$. A subgroup $H \leq G$ is \emph{suitable} for this action if it acts non-elementarily and does not normalize any finite normal subgroup of $G$.
\end{definition}

Note that if $G$ has a suitable subgroup then necessarily $K(G) = \{ 1 \}$. In most of our applications $G$ will be torsion-free, and therefore a subgroup is suitable if and only if it is non-elementary. Note also that the notion of a suitable subgroup depends on the relative generating set in general, but if $(G,\{H_\lambda\}_{\lambda\in \Lambda})$ is a relatively hyperbolic pair, one can characterize suitable subgroups purely in terms of the pair $(G,\{H_\lambda\}_{\lambda\in \Lambda})$:

\begin{lemma}[{\cite[Lemma 3.22]{chifan2023small}}]\label{lem. suitable for rh}
    Let $(G,\{H_\lambda\}_{\lambda\in \Lambda})$ be a relatively hyperbolic pair and let $X\subset G$ be any finite relative generating set with respect to $\{H_\lambda\}_{\lambda\in \Lambda}$. Then a subgroup $K\leq G$ is suitable with respect to the action $\Gamma\left(G,X\sqcup \left(\bigsqcup_{\lambda\in\Lambda}H_\lambda\right)\right)$ if and only if $K$ is not virtually cyclic, contains a \emph{hyperbolic element} (i.e., an infinite-order element that does not conjugate into any of $H_\lambda$), and does not normalize any non-trivial finite subgroup.
\end{lemma}

So we can unambiguously speak of a \emph{suitable subgroup with respect to the relatively hyperbolic pair} $(G,\{H_\lambda\}_{\lambda\in \Lambda})$, meaning a suitable subgroup with respect to any finite relative generating set of $\{H_\lambda\}_{\lambda\in \Lambda}$ in $G$.

If $g \in G$ is a loxodromic element for an acylindrical action, then $g$ admits an \emph{elementary closure}, denoted $E(g)$, i.e. a unique maximal virtually cyclic overgroup.

\subsubsection{Isolated components}\label{sec.ic}

Let $p$ be a path in $\Gamma(G,\A)$. The \emph{label} of $p$, denoted $\lab(p)$, is obtained by concatenating all labels of the edges in $p$ and is a word over $\A$. The length of $p$ is denoted by $\ell_X(p)$, and the initial (resp. terminal) vertex of $p$ is denoted by $p^-$ (resp. $p^+$).
For $\lambda\in\Lambda$, let $\widehat{d}_{\lambda}$ be the relative metric on $H_{\lambda}$ with respect to $X$. The following terminology goes back to \cite{osin2006relatively}.

\begin{definition}\label{connect}
Let $p$ be a path in $\Gamma(G,\A)$. For every $\lambda\in\Lambda$, an $H_{\lambda}$\emph{-subpath} $q$ of $p$ is a nontrivial subpath of $p$ such that $\lab(q)$ is a word over the alphabet $H_{\lambda}$ (if $p$ is a cycle, we allow $q$ to be a subpath of some cyclic shift of $p$). An $H_{\lambda}$-subpath $q$ of $p$ is an $H_{\lambda}$-\emph{component} if $q$ is not properly contained in any other $H_{\lambda}$-subpath. Two $H_{\lambda}$-components $q_1$ and $q_2$ of $p$ are \emph{connected} if for any two vertices $v_1\in q_1,v_2\in q_2$, there exists an edge $t$ in $\Gamma(G,\A)$ such that $t^-=v_1,t^+=v_2$, and $\lab(t)$ is a letter from $H_{\lambda}$. An $H_{\lambda}$-component $q$ of $p$ is \emph{isolated} if it is not connected to any other $H_{\lambda}$-component of $p$. Below, the $H_\lambda$-components will be collectively called \emph{components}. (By contrast, the maximal connected subspaces of a topological space will be called \emph{connected components}.)
\end{definition}

\begin{remark}
    The definition of connectedness in \cite{dahmani2017hyperbolically} is seemingly weaker than the version above. \cite[Definition 4.5]{dahmani2017hyperbolically} instead requires the existence of a path $t_1$ connecting a vertex of $q_1$ with a vertex of $q_2$ with label a word over $H_\lambda$. However, the two definitions are actually equivalent. Suppose that there exists a path $t_1$ that satisfies \cite[Definition 4.5]{dahmani2017hyperbolically}. Let $v_1\in q_1,v_2\in q_2$ be any vertices, and let $t_2$ (resp. $t_3$) be a subpath of $q_1$ or $q^{-1}_1$ (resp. $q_2$ or $q^{-1}_2$) from $v_1$ to $t^-_1$ (resp. from $t^+_1$ to $v_2$). The concatenation $t_2t_1t_3$ is a path from $v_1$ to $v_2$ with label a word over $H_\lambda$. Recall that $v_1$ and $v_2$ are elements of $G$. Then we have $v^{-1}_1v_2\in H_\lambda$. So there exists an edge $t$ from $v_1$ to $v_2$ with $\lab(t)\in H_\lambda$.
\end{remark}

\begin{proposition}[{\cite[Proposition 4.14]{dahmani2017hyperbolically}}]\label{prop. total length of i.c. in g.p.}
If $\{H_\lambda\}_{\lambda\in\Lambda}\hookrightarrow_{h}(G,X)$, then there exists a number $D>0$ satisfying the following property. Let $p$ be an $n$-gon in $\Gamma(G,\A)$ with geodesic sides $p_1,\ldots,p_n$ and let $I$ be a subset of the set of sides of $p$ such that every side $p_i\in I$ is an isolated $H_{\lambda_i}$-component of $p$ for some $\lambda_i\in\Lambda$. Then
$$\sum_{p_i\in I}\widehat{\ell}_{\lambda_i}(p_i)\leq Dn.$$
\end{proposition}

\begin{lemma}[{\cite[Lemma 4.21]{dahmani2017hyperbolically}}]\label{lem. consecutive components}
Suppose $\{H_{\lambda}\}_{\lambda\in\Lambda}\hookrightarrow_h(G,X)$. Let $W$ be the set consisting of all words $w \in \A^\ast$ such that
\begin{enumerate}
    \item[(W1)] $w$ contains no subwords of type $xy$, where $x,y\in X$;
    \item[(W2)] if $w$ contains a letter $h\in H_{\lambda}$ for some $\lambda\in \Lambda$, then $\widehat{d}_{\lambda}(1,h)>50D$, where $D$ is given by Proposition \ref{prop. total length of i.c. in g.p.};
    \item[(W3)] if $h_1xh_2$ (resp. $h_1h_2$) is a subword of $w$, where $x\in X,h_1\in H_{\lambda},h_2\in H_{\mu}$, then either $\lambda\neq \mu$ or the element represented by $x$ in $G$ does not belong to $H_{\lambda}$ (resp. $\lambda\neq \mu$).
\end{enumerate}
Then the following hold.
\begin{enumerate}
    \item[(a)] Every path in the Cayley graph $\Gamma(G,\A)$ labeled by a word from $W$ is a $(4,1)$-quasi-geodesic.
    \item[(b)] If $p$ is a path in $\Gamma(G,\A)$ labeled by a word from $W$, then for every $\lambda\in\Lambda$, every $H_{\lambda}$-component of $p$ is isolated.
    \item[(c)] For every $\eps>0$, there exists $R>0$ satisfying the following condition. Let $p,q$ be two paths in $\Gamma(G,\A)$ such that 
    $$\ell_{X\sqcup \mathcal{H}}(p)\geq R,~~~~\lab(p),\lab(q)\in W,$$
    and $p,q$ are oriented $\eps$-close, i.e.,
    $$\max\{d_{\A}(p^-,q^-),d_{\A}(p^+,q^+)\}\leq \eps,$$
    where $d_{\A}$ is the combinatorial metric of $\Gamma(G,\A)$.
    Then there exist five consecutive components of $p$ which are respectively connected to five consecutive components of $q$. In other words,
    $$p=ra_1x_1a_2x_2a_3x_3a_4x_4a_5s,~~~~q=tb_1y_1b_2y_2b_3y_3b_4y_4b_5u,$$
    such that the following hold.
    \begin{enumerate}
        \item[(i)] $r$ (resp. $t$) is a subpath of $p$ (resp. $q$) whose label ends with a letter from $X$.
        \item[(ii)] $s$ (resp. $u$) is a subpath of $p$ (resp. $q$) whose label starts with a letter from $X$.
        \item[(iii)] For $i=1,\ldots,4$, $x_i$ and $y_i$ are either trivial subpaths or subpaths labeled by a letter over $X$; 
        \item[(iv)] For $i=1,\ldots,5$, $a_i$ and $b_i$ are components connected to each other.
    \end{enumerate}
\end{enumerate}
\end{lemma}

\begin{remark}
Conclusion (b) of Lemma \ref{lem. consecutive components} is not stated in \cite[Lemma 4.21]{dahmani2017hyperbolically}, but it is proved in the second paragraph of the proof of \cite[Lemma 4.21]{dahmani2017hyperbolically}.
\end{remark}

\subsubsection{Small cancellation}\label{subsec. sc}
We recall the small cancellation theory of Ol'shanski\u{\i} \cite{olshanskii1993residualing}. We will use another small cancellation condition that is easier to establish \cite{chifan2023small}.

\begin{definition}
A symmetric set of words $\mathcal R \subset \A^\ast$ satisfies the \emph{$C(\eps , \mu , \rho )$--condition} for some $\eps \geq 0$ and $\mu, \rho >0$,
if the following conditions hold.
\begin{enumerate}
\item[(a)] All words in $\mathcal R$ are $\Gamma(G, \A)$-geodesic and have length at least  $\rho $.
\item[(b)] Suppose that words $R,R^\prime \in \mathcal R$ have initial subwords $U$ and $U^\prime$, respectively,  such that
\begin{equation}\label{piece1}
\max \{ \| U\| ,\, \| U^\prime\| \} \geq \mu \min\{\| R\|, \| R^\prime\|\}
\end{equation}
and $U^\prime = YUZ$ in $G$ for some words $Y, Z$ of length
\begin{equation}\label{piece2}
\max \{ \| Y\| , \,\| Z\| \} \leq \eps.
\end{equation}
Then $YRY^{-1}=R^\prime$ in $G$.
\end{enumerate}
Further, we say that $\mathcal R$ satisfies the \emph{$C_1(\eps , \mu , \rho )$--condition} if,  in addition to (a) and (b), we have the following.
\begin{enumerate}
\item[(c)] Suppose that a word $R\in \mathcal R$ contains two disjoint subwords $U$ and $U^\prime$ such that $U^\prime = YUZ$ or $U^\prime = YU^{-1}Z$ in $G$ for some words $Y$, $Z$ and the inequality (\ref{piece2}) holds. Then
    $$ \max \{ \| U\| ,\, \| U^\prime\| \}< \mu \| R\|.$$
\end{enumerate}
\end{definition}

\begin{lemma}[{\cite[Lemma 3.26]{chifan2023small}}]\label{lem. hull}
Let $G$ be a group, $\{H_\lambda\}_{\lambda\in\Lambda}$ a collection of subgroups of $G$ such that $\{H_\lambda\}_{\lambda\in\Lambda}\hookrightarrow_h (G,X)$ for some $X\subseteq G$, and let $\A \coloneqq X\sqcup(\bigsqcup_{\lambda\in\Lambda}H_\lambda)$. Then for any $r\geq 1$, there exist $\eps\geq 0$ and $\mu, \rho>0$ such that, for any finite symmetric set of words $\mathcal R \subset \A^\ast$ satisfying $C(\eps, \mu , \rho )$, the following hold.
\begin{enumerate}
\item[(a)] The restriction of the natural homomorphism $\pi \colon G\to \bg \coloneqq G/\ll \mathcal{R}\rr$ to the set
\begin{equation}\label{BN}
B_r= \{ g\in G\mid |g|_\mathcal A\leq r\}
\end{equation}
is injective. In particular, the restriction of $\pi$ to $\bigcup\limits_{\lambda\in\Lambda} H_\lambda$ is injective.
\item[(b)] $\{\pi(H_\lambda)\}_{\lambda\in \Lambda}\hookrightarrow_h (\bg, \pi(X))$.
\item[(c)] For each $\bar g\in \bg$ of finite order, there exists $g\in G$ of finite order such that $\bar g=\pi(g)$.
\end{enumerate}
\end{lemma}

\begin{lemma}[{\cite[Lemma 3.5]{chifan2023small}}]\label{lem. centralizer 0}
Suppose $\{H_\lambda\}_{\lambda \in \Lambda} \hookrightarrow_h (G, X)$. Let $\A = X \sqcup \H$, and suppose that $\Gamma (G, \mathcal A)$ is hyperbolic. For any $r\ge 1$, there exist $\eps,\rho>0$ such that the following holds.

Let $\mathcal R \subset \mathcal A^\ast$ satisfy the $C_1(\eps, 1/100, \rho)$ small cancellation condition, and let $\pi \colon G \to \overline G \coloneqq G/ \ll \mathcal R \rr.$ For any $g\in G$ of length $|g|_\mathcal A\le r$, we have $C_{\bg}(\pi(g))=\pi(C_G(g))$.
\end{lemma}

This small cancellation condition is often hard to establish, so we use a different small cancellation condition that implies it and is easier to check. Below, we say that two letters $a$, $b$ of a word $W\in \mathcal A^\ast$ are \emph{cyclically consecutive} if they are consecutive or if $a$ (respectively, $b$) is the last (respectively, first) letter of $W$.

\begin{definition}\label{def. W small cancellation}
    A set of words $\W\subset \mathcal A^\ast$ satisfies the \emph{$W(\xi, \sigma)$ condition} for some $\xi, \sigma \geq 0$ if the following hold:

    \begin{enumerate}
    \item[(SC1)] If $a\in H_\lambda$ and $b\in H_\mu$ are cyclically consecutive letters of some word from $\W$, then $H_\lambda\cap H_\mu=\{1\}$.
    \item[(SC2)] If a letter $a\in H_\lambda$ occurs in some word from $\W$, then $\widehat{d}_\lambda(1, a) \geq \xi $.
    \item[(SC3)] For each letter $a\in \H$, there is at most one occurrence of $a^{\pm 1}$ in all words from $\W$. More precisely, let $W,V\in\W$. Suppose that $W\equiv W_1aW_2$ and $V\equiv V_1a^\eps V_2$ for some $a\in \H$, $W_1, W_2, V_1, V_2\in \A^\ast$, and $\eps=\pm 1$. Then $\eps=1$, and $W_i\equiv V_i$ for $i=1,2$; in particular, $W\equiv V$.
    \item[(SC4)] For every $W\in \W$, we have $\| W\| \geq \sigma$.
    \end{enumerate}
\end{definition}

By abuse of notation, for any word $W\in \A^\ast$, we will use $W$ to denote the element of $G$ represented by $W$. Further, for a set of words $\W\subset \A^\ast$, we will use $\ll \W \rr_G$ to denote the normal closure of the elements of $G$ represented by the words in $\W$.

\begin{lemma}[{\cite[Lemma 3.16 (a)]{chifan2023small}}]\label{lem. W->C}
For any positive constants $\eps$, $\mu$, and $\rho$, there exist positive $\xi$ and $\sigma$ such that, for any set of words $\W=\{ W_j\}_{j\in J}\subseteq \A^\ast$ satisfying $W(\xi, \sigma)$,  the symmetrization of $\mathcal W$ satisfies  $C_1(\eps,\mu,\rho)$.
\end{lemma}

Combining these results we obtain:

\begin{lemma}\label{lem. sc}
    For any $r\geq 1$, there exist $\xi,\sigma>0$ such that the following holds: Let $\W\subset \A^\ast$ be any finite set of words that satisfies $W(\xi,\sigma)$, let $\bg:=G/\ll \W \rr$ and let $\pi\colon G\to \bg$ be the natural homomorphism. Then the following hold:

    \begin{enumerate}
    \item The restriction of $\pi$ to the set
     \[
     B_r= \{ g\in G\mid |g|_\A\leq r\}
     \]
     is injective. In particular, the restriction of $\pi$ to $\bigcup_{\lambda\in\Lambda} H_\lambda$ is injective.
     \item $\{\pi(H_\lambda)\}_{\lambda\in \Lambda}\hookrightarrow_h (\bg, \pi(X))$.
     \item For every $g\in G$ with $|g|_{\A}\leq r$, it holds that $C_{\bg}(\pi(g))=\pi(C_G(g))$.
     \item For each $\bar g \in \bg$ of finite order, there exists $g\in G$ of finite order such that $\bar g=\pi(g)$.
    \end{enumerate}
\end{lemma}

\begin{proof}
     Fix $r\geq 1$. Let $\eps_1,\rho_1>0$ be the constants given by \Cref{lem. centralizer 0}. Let $\eps_2,\mu_2,\rho_2>0$ be the constants given by \Cref{lem. hull}. Let $\eps=\max\{\eps_1,\eps_2\},\mu=\min\{\mu_2,1/100\},\rho=\max\{\rho_1,\rho_2\}$, and let $\xi,\sigma>0$ be the constants given by \Cref{lem. W->C}.
\end{proof}

\subsection{Dehn filling and Cohen--Lyndon property}

The most general setting of (group theoretic) Dehn filling consists of a group $G$, a subgroup $H\leq G$ and a normal subgroup $N\lhd H$. The Dehn filling process produces a quotient $G/\ll N \rr_G$, where $\ll\cdot\rr_G$ indicates the normal closure of a subset in $G$. To prove useful results, in practice it is often assumed that $G$ and $H$ satisfy certain negative-curvature conditions, such as $G$ being hyperbolic relative to $H$ \cite{osin2007peripheral,groves2008dehn} or more generally $H$ being hyperbolically embedded \cite{dahmani2017hyperbolically}.

\begin{theorem}[{\cite[Theorem 1.1]{osin2007peripheral}}]\label{thm. osin df}
    Let $G$ be a group that is hyperbolic relative to a subgroup $H$. Then there exists a finite subset $F\subset H\smallsetminus\{1\}$ such that if a normal subgroup $N\lhd H$ satisfies $N\cap F=\emptyset$, then $\pi$ maps $H$ to a subgroup of $\pi(G)$ isomorphic to $H/N$ and $\pi(G)$ is hyperbolic relative to $\pi(H)$, where $\pi\colon G\to G/\ll N \rr_G$ is the natural quotient map.
\end{theorem}

Combining this with \cite[Corollary 2.41]{osin2006relatively}, we obtain:

\begin{corollary}\label{cor. osin df}
    Let $G$ be a group that is hyperbolic relative to a subgroup $H$. Then there exists a finite subset $F\subset H\smallsetminus\{1\}$ such that if a normal subgroup $N\lhd H$ satisfies that $N\cap F=\emptyset$ and $H/N$ is hyperbolic, then $\pi(G)$ is hyperbolic.
\end{corollary}

Our method to control homology is via the notion of a Cohen--Lyndon triple, which was first studied by Cohen and Lyndon for free groups \cite{cohen1963free}, hence the name.

\begin{definition}
    Let $G\geq H\rhd N$ be groups. The triple $(G,H,N)$ is called a \emph{Cohen--Lyndon triple} if there exists a left transversal $T$ of $H\ll N \rr_G$ in $G$ such that $\ll N \rr_G$ decomposes as a free product:

    \[\ll N \rr_G = \Asterisk_{t\in T} tNt^{-1}.\]
\end{definition}

\begin{theorem}[{\cite[Theorem 2.5]{sun2018cohomologyi}}]\label{thm. sun cl}
    Let $G$ be a group with a hyperbolically embedded subgroup $H$. Then there exists a finite subset $F\subset H\smallsetminus\{1\}$ such that if a normal subgroup $N\lhd H$ satisfies $N\cap F=\emptyset$, then $(G,H,N)$ is a Cohen--Lyndon triple.
\end{theorem}

\begin{lemma}[{\cite[Lemma 4.22]{chifan2023wreath}}]\label{lem. cl transitive}
    Let $G\geq H\geq K$ be groups. Suppose that $(G,H, \ll K \rr_H)$ and $(H,K,K)$ are Cohen--Lyndon triples. Then $(G,K,K)$ is a Cohen--Lyndon triple.
\end{lemma}

The Cohen--Lyndon property allows strong control of the geometry of the corresponding quotient. In the following theorem, if $(G, H, N)$ is a Cohen--Lyndon triple, we denote by $\bg \coloneqq G/\ll N \rr_G$ and $\bh \coloneqq H/N$.

\begin{theorem}[{\cite[Theorem 1.12]{petrosyan2024l2}}]
\label{thm. df space}
    Let $(G,H,N)$ be a Cohen--Lyndon triple. Let $BG$ (resp. $BH,B\bh$) be a $K(G,1)$ (resp. $K(H,1),K(\bh,1)$) CW-complex. Let $\phi\colon BH\rightarrow BG$ be a cellular map induced by the inclusion $H\hookrightarrow G$. Let $\psi\colon BH\rightarrow B\bh$ be a cellular map induced by $H\twoheadrightarrow \bh$. Let $X$ be the CW-complex obtained by gluing the mapping cylinders $M_\phi$ and $M_\psi$ along their common subcomplex $BH$. Then $X$ is a $K(\bg,1)$.
\end{theorem}

We will only be concerned with an especially easy case.

\begin{corollary}
\label{cor. classifying space}
    Let $(G, H, H)$ be a Cohen--Lyndon triple such that $H \cong \Z$. Then a $K(\bg, 1)$ can be obtained from a $K(G, 1)$ by attaching a 2-cell along a generator of $H$.
\end{corollary}

\begin{proof}
    We apply Theorem \ref{thm. df space}. Let $BG$ be a $K(G, 1)$, let $BH$ be a circle (a $K(H, 1)$) and $B\bh$ a point (a $K(\bh, 1)$). Then $\phi \colon BH \to BG$ is the loop representing the generator of $H$ in $G$, so $M_\phi$ deformation retracts onto $BG$. Moreover $\psi \colon BH \to B \bh$ is the unique map from a circle to a point, and therefore $M_\psi$ is a disk with boundary $BH$. Gluing $M_\phi$ and $M_\psi$ along their common subcomplex $BH$ gives a complex $X$ that is homotopy equivalent to the complex obtained by gluing a disk along the loop representing the generator of $H$ in $G$.
\end{proof}

\section{A small cancellation theorem}
\label{sec:inductive}

In this section we prove a small cancellation theorem, which allows us to take quotients imposing strong conditions and preserving many useful properties, in particular homological ones. This is achieved by ensuring that the relators we add are both Cohen--Lyndon and satisfy the $W(\xi, \sigma)$ small cancellation condition. We will use Theorem \ref{thm. inductive step} for the inductive step of all of our constructions.

\begin{theorem}
\label{thm. inductive step}
    Suppose that $G$ is a group with a hyperbolically embedded finite family of proper subgroups $\{H_i\}^s_{i=1}\hookrightarrow_h G$. Let $X\subset G$ be a subset such that $\{H_i\}^s_{i=1}\hookrightarrow_h(G,X)$ and the action $G\curvearrowright \Gamma(G,X\sqcup (\bigsqcup^s_{i=1}H_i))$ is acylindrical. Let $K,K'\leq G$ be suitable subgroups with respect to the action $G\curvearrowright \Gamma(G,X\sqcup (\bigsqcup^s_{i=1}H_i))$, let $g_1,\dots,g_N\in G$, and let $r\geq 1$. Then there exists a quotient $\pi\colon G\to \bg$ with the following properties.
    \begin{enumerate}[label=(\roman*)]
    \item\label{item. hyperbolicity} $\{\pi(H_i)\}^s_{i=1}\hookrightarrow_h(\bg,\bx)$. In particular, if $(G,\{H_i\}^s_{i=1})$ is a relatively hyperbolic pair, then so is $(\bg,\{\pi(H_i)\}^s_{i=1})$.
    \item\label{item. torsion element} For each $\bar g\in \bg$ with finite order, there exists $g\in G$ with finite order such that $\bar g=\pi(g)$. In particular, if $G$ is torsion-free then so is $\bg$.
    \item\label{item. injectivity} $\pi$ maps $B_{X\sqcup (\bigsqcup^s_{i=1}H_i)}(r)$ injectively onto $B_{\bx\sqcup (\bigsqcup^s_{i=1}\pi(H_i))}(r)$.
    \item\label{item. centralizer} For every $g\in G$ with $|g|_{X\sqcup (\bigsqcup^s_{i=1}H_i)}\leq r$, it holds $C_{\bg}(\pi(g))=\pi(C_G(g))$.
    \item\label{item. K absorb} $\pi(g_n) \in \pi(K)$ for $n = 1, \ldots, N$.
    \item\label{item. K non-elementary} There exists $\overline X\subset \bg$ such that $\{\pi(H_i)\}^s_{i=1}\hookrightarrow_h(\bg,\overline X)$; the action $G\curvearrowright \Gamma(\bg, \overline X\sqcup (\bigsqcup^s_{i=1}\pi(H_i)))$ is acylindrical; and $\pi(K)$ and $ \pi(K')$ are suitable with respect to this action. If $|X|<\infty$, then we can take $\overline X=\pi(X)$.
    \item\label{item. gd} A $K(\bg,1)$ can be obtained from a $K(G,1)$ by attaching $n$ 2-cells, and therefore $\gd(\bg)\leq \max\{2,\gd(G)\}$.
    \item\label{item. homology} For all $\bg$-modules $A$ and all $j \geq 3$ the induced maps $H_j(G; A) \to H_j(\bg; A)$ and $H^j(\bg;A)\to H^j(G;A)$ are isomorphisms; $H_2(G; A) \to H_2(\bg;A)$ is injective; and $H^2(\bg; A) \to H^2(G; A)$ is surjective.
    \end{enumerate}
\end{theorem}

This entire section is devoted to the proof of Theorem \ref{thm. inductive step}. We will prove the theorem for $n=1$; the general case follows by induction. Let $X_1=X\cup\{g_1,g^{-1}_1\}$. Note that we have $\{H_i\}^s_{i=1}\hookrightarrow_h(G,X_1)$ by \Cref{lem. change relative generating set} and the action $G\curvearrowright\Gamma(G,X_1\sqcup\{H_i\}^s_{i=1})$ is acylindrical by \cite[Proposition 4.5]{abbott2017hyperbolic}.

By \cite[Lemma 5.6]{hull2013small}, the suitable subgroup $K$ contains pairwise non-commensurable loxodromics $k_q:q=1,\dots, 25$ such that $E(k_q)=\langle k_q \rangle$. At least twenty four of these elements, say $k_1,\dots,k_{24}$, are not commensurable with $g_1$. Similarly, we can find non-commensurable loxodromics $k'_1,k'_2\in K'$, each of which is not commensurable with any one of $g_1,k_1,\dots, k_{24}$, such that $E(k'_1)=\langle k'_1 \rangle, E(k'_2)=\langle k'_2 \rangle$.

\begin{lemma}
We have a hyperbolic embedding
    \begin{equation}\label{eq. he 1}
        \{H_1,\dots,H_s,\langle k_1 \rangle,\dots,\langle k_{24}\rangle, \langle k'_1\rangle ,\langle k'_2\rangle\}\hookrightarrow_h(G,X_1).
    \end{equation}
\end{lemma}

\begin{proof}
    Consider the hyperbolic embedding $\{H_i\}^s_{i=1}\hookrightarrow_h(G,X_1)$
    and the pairwise non-commensurable loxodromics $k_1,\dots,k_{24},k'_1,k'_2$. In the proof of \cite[Theorem 6.8]{dahmani2017hyperbolically}, it is verified that the family $\{\langle k_1 \rangle,\dots,\langle k_{24}\rangle, \langle k'_1\rangle ,\langle k'_2\rangle\}$ satisfies condition (C3) of \Cref{thm. criterion for he}. As this family clearly satisfies (C1) and (C2), \Cref{thm. criterion for he} provides the desired result.
\end{proof}

For some positive integer $u$ that will be specified later, let

\begin{equation}\label{eq. g1...g7}
    h_m:=k^u_{3m-2}k^u_{3m-1}k^u_{3m},\;\; m=1,\dots,7,
\end{equation}
and let

\begin{equation}\label{eq. g8}
    h_8:=k^u_{22}k^u_{23}g_1k^u_{24}.
\end{equation}

\begin{lemma}\label{lem. relative hyperbolic}
There exists $U>0$ such that if $u>U$, then the set $\{h_1,\dots,h_8\}$ freely generates a free group $F_8$ in $G$ and 

\begin{equation}\label{eq. 1}
    \{F_8,H_1,\dots,H_s,\langle k_1 \rangle,\dots,\langle k_{24}\rangle, \langle k'_1\rangle ,\langle k'_2\rangle \}\hookrightarrow (G,X_1).
\end{equation}
\end{lemma}

\begin{proof}
    For $q=1,\dots, 24$, let $\widehat{d}_q\colon \langle k_q\rangle \times \langle k_q\rangle\rightarrow [0,\infty]$ be the relative metric corresponding to the hyperbolic embedding \eqref{eq. he 1}. Recall that these metrics are locally finite. So there exists $U>0$ such that for all $u>U$ we have
    
    \[\widehat{d}_q(1,k^u_q)>50D\]
    for $q=1,\dots,24$, where $D$ is given by Proposition \ref{prop. total length of i.c. in g.p.}. Let

    \[\B:=X_1\sqcup \langle k'_1\rangle\sqcup \langle k'_2\rangle \sqcup \left(\bigsqcup^s_{i=1}H_i\right) \sqcup \left(\bigsqcup^{24}_{q=1}\langle k_q\rangle\right).\]
    Below, we will think of $h_1,\dots,h_8$ as words over the alphabet $\B$.
    
    Let $\W$ be the set of reduced words over the alphabet $\{h^{\pm 1}_1,\dots,h^{\pm 1}_8\}$. For each $W\in\W$, by substituting each letter in $W$ by a word over the alphabet $\B$ using \eqref{eq. g1...g7} and \eqref{eq. g8}, we obtain a word $\theta(W)$. Let 
    
    \[\V=\{\theta(W)\mid W\in \W\}.\]
    Note that each word of $\V$ represents an element of $F_8$ and satisfies (W1), (W2) and (W3) of \Cref{lem. consecutive components} with respect to the hyperbolic embedding \eqref{eq. he 1}.

    \begin{claim}
       The group $F_8$ is free on basis $\{h_1,\dots,h_8\}$ for $u>U$.
    \end{claim}

    \begin{proof}[Proof of the claim]
        Let $W$ be any non-empty word in $\W$. The word $\theta(W)$ labels a path $p$ in $\Gamma(G,\B)$. Suppose that $W=1$ in $G$. Then $p$ labels a geodesic polygon with at most $4\|W\|$ sides and $3\|W\|$ components. By \Cref{lem. consecutive components}, each of these components is isolated. Then \Cref{prop. total length of i.c. in g.p.} implies

        \[3\|W\|\cdot (50D)<4\|W\|D,\]
        which is absurd. So $W\neq 1$ in $G$.
    \end{proof}

    We will apply \Cref{thm. criterion for he} to $F_8$ with respect to the Cayley graph $\Gamma(G,\B)$. The group $F_8$ obviously satisfies (C1). Let us verify (C2). For any $x\in F_8$, the path between $1$ and $x$ is labeled by a word in $\V$. By \Cref{lem. consecutive components}, this path is a $(4,1)$-quasi-geodesic, which implies (C2) for $F_8$.

    \medskip

    The rest of the proof of Lemma \ref{lem. relative hyperbolic} is devoted to verifying (C3). Fix $\eps>0$ and $g\in G$. Let $R>0$ be the constant given by \Cref{lem. consecutive components} and suppose

    \[\diam(F_8\cap (gF_8)^{+\eps})\geq R.\]

    Recall that we use $p^{\pm}$ to denote the initial and terminal vertex of  a combinatorial path $p$. The previous equation gives oriented $\eps$-close paths $p,p'\in \Gamma(G,\B)$ with $p^- = 1, (p')^- = g$ such that $\ell(p)\geq R$ and $\lab(p),\lab(p')\in \V$. By \Cref{lem. consecutive components}, there exist five consecutive components of $p$ which are respectively connected to five consecutive components of $p'$. In particular, there exist consecutive components $c_1,c_2$ of $p$ and consecutive components $c'_1,c'_2$ of $p'$ such that 
    
    \begin{enumerate}[label=(\Alph*)]
        \item $c_i$ is connected to $c'_i$ for $i=1,2$; and
        \item\label{item. terminal letter} $\lab(p_1)\in \V$, where $p_1$ is the initial subpath of $p$ such that $p^+_1=c^+_1$.
    \end{enumerate}

    Let $p_2$ be the minimal initial subpath of $p$ such that $p_2$ properly contains $p_1$ and $\lab(p_2)\in \V$. There exist words $W_1,W_2\in \W$ such that
    \[\lab(p_1)=\theta(W_1),\;\;\lab(p_2)=\theta(W_2).\]
    We can write $W_1,W_2$ as
    \[W_1=h^{\eta_1}_{m_1}\dots h^{\eta_v}_{m_v},\;\; W_2=W_1h^{\eta_{v+1}}_{m_{v+1}},\]
    with $m_1,\dots,m_{v+1}\in\{1,\dots,8\},$ and $\eta_1,\dots,\eta_{v+1}\in\{1,-1\}$. There exist $q_1,q_2\in\{1,\dots,24\}$ such that 
    \[\lab(c_1)\in \langle k_{q_1} \rangle,\;\; \lab(c_2)\in\langle k_{q_2} \rangle.\]
    As $c_1,c_2$ are consecutive components of $p$, by combining \eqref{eq. g1...g7}, \eqref{eq. g8} and Item \ref{item. terminal letter}, we obtain:
    \begin{enumerate}[label=(\Alph*)]
    \setcounter{enumi}{2}
        \item\label{item. c1 c2} $k_{q_1}$ (resp. $k_{q_2}$) is either the initial or terminal letter of $h_{m_v}$ (resp. $h_{m_{v+1}}$). Moreover, $(c_1)^+=(c_2)^-$. As $W_2$ is a reduced word, we also have $q_1\neq q_2$.
    \end{enumerate}
    Whether it is the initial or terminal letter depends on the exponents $\eta_v, \eta_{v+1} = \pm 1$. Let us stress that Lemma \ref{lem. consecutive components} only guarantees that $c_1$ and $c_2$ are consecutive in the sense that they may be separated only by a word over $X$, but the choice of our words prevents letters in $X$ from separating $c_1$ and $c_2$.
    
    Let $p''$ be the initial subpath of $p'$ such that $(p'')^+=(c'_1)^+$, let $p'_1$ be the maximal initial subpath of $p''$ such that $\lab(p'_1)\in \V$, and let $p'_2$ be the minimal initial subpath of $p'$ such that $p'_2$ properly contains $p'_1$ and $\lab(p'_2)\in \V$. Again, there exist words $W'_1,W'_2\in \W$ such that 
    \[\lab(p'_1)=\theta(W'_1),\;\; \lab(p'_2)=\theta(W'_2).\]
    We can write $W'_1,W'_2$ as

    \[W'_1=h^{\eta'_1}_{m'_1}\dots h^{\eta'_w}_{m'_w},\;\; W'_2=W'_1h^{\eta'_{w+1}}_{m'_{w+1}},\]
    with $m'_1,\dots,m'_{w+1}\in\{1,\dots,8\},$ and $\eta'_1,\dots,\eta'_{w+1}\in\{1,-1\}$. As $c_1,c_2$ and $c'_1,c'_2$ are respectively connected, we have

    \[\lab(c'_1)\in \langle k_{q_1} \rangle,\;\; \lab(c'_2)\in\langle k_{q_2} \rangle.\]

    \begin{claim}\label{cl. p1 c1}
        $(p'_1)^+=(c'_1)^+$.
    \end{claim}

    \begin{proof}[Proof of the claim]
        Suppose that $(p'_1)^+\neq (c'_1)^+$. Note that we also have $(p'_2)^+\neq (c'_1)^+$, as otherwise $p'_2$ would be the maximal initial subpath of $p''$ with $\lab(p_2)\in \V$ and thus would equal $p'_1$, contradiction. Similarly, one can prove that $p'_2$ must contain $c'_1$.

        Therefore, $k_{q_1}$ is a letter of $h_{m'_{w+1}}$, and from Item \ref{item. c1 c2} and Equations \eqref{eq. g1...g7} and \eqref{eq. g8}, we see that $k_{q_1}$ must be either the first or the last letter of $h_{m'_{w+1}}$.

        As $c'_1$ and $c'_2$ are consecutive components, the path $p'_2$ contains $c'_1$, and $(p'_2)^+\neq (c'_1)^+$, from Equations \eqref{eq. g1...g7} and \eqref{eq. g8} we get that $p'_2$ must also contain $c'_2$. So $k_{q_2}$ is also a letter of $h_{m'_{w+1}}$. Using equations \eqref{eq. g1...g7} and \eqref{eq. g8} once again, we see that $k_{q_2}$ cannot be the first or last letter of $h_{m'_{w+1}}$, which contradicts Item \ref{item. c1 c2}.
    \end{proof}

    \begin{claim}\label{cl. not separated}
        $(c'_1)^+=(c'_2)^-$.
    \end{claim}

    \begin{proof}[Proof of the claim]
        Suppose $(c'_1)^+\neq (c'_2)^-$. As $c'_1$ and $c'_2$ are consecutive components, there is an edge $p'_3$ of $p'$ such that $(p'_3)^-=(c'_1)^+, (p'_3)^+=(c'_2)^-$ and $\lab(p'_3)\in X_1$. From \eqref{eq. g1...g7} and \eqref{eq. g8} we get that $\lab(p'_3)$ is either $g_1$ or $g^{-1}_1$. When combined with \Cref{cl. p1 c1}, this implies that $g_1$ is either the first or the last letter of $h_{m'_{w+1}}$, contradiction.
    \end{proof}
    
    As $c_1$ and $c'_1$ are connected, by definition there exists an edge $e$ from $c^+_1$ to $(c'_1)^+$ labeled by an element of $\langle k_{q_1} \rangle$. Then
    \[(c^+_1)^{-1}(c'_1)^+\in \langle k_{q_1} \rangle.\]
    Similarly,
    \[(c^+_2)^{-1}(c'_2)^+\in \langle k_{q_2} \rangle.\]

    From \Cref{cl. not separated} and Item \ref{item. c1 c2}, we obtain $(c_1^+)^{-1}(c'_1)^+=(c_2^+)^{-1}(c'_2)^+$. So
    \[(c_1^+)^{-1}(c'_1)^+\in \langle k_{q_1} \rangle\cap \langle k_{q_2} \rangle.\]
    As $k_{q_1}$ and $k_{q_2}$ are non-commensurable, the latter intersection is $\{1\}$. So 
    \[p^+_1=c^+_1=(c'_1)^+=(p'_1)^+.\]
    Thus, we can concatenate the paths $p_1$ and $(p'_1)^{-1}$ to get a path from $1$ to $(p'_1)^-=g$. The word $\lab(p_1)(\lab(p'_1))^{-1}$ thus represents $g$ in $G$. Each of $\lab(p_1)$ and $\lab(p'_1)$ represents an element in $F_8$.
    
    So $g\in F_8$, which establishes (C3).
    Therefore \eqref{eq. 1} follows from \Cref{thm. criterion for he}, and this concludes the proof of Lemma \ref{lem. relative hyperbolic}.
\end{proof}

Below, we will fix an $u>U$, and the elements $h_1,\dots,h_8$ will correspond to this choice of $u$. Let $F_2<F_8$ be the free subgroup generated by $\{h_7,h_8\}$. Note that $F_8$ is hyperbolic relative to the family $\{\langle h_1 \rangle,\dots,\langle h_6\rangle,F_2\}$. For simplicity, denote
\[\mathbb{S}:=\{\langle h_1 \rangle,\dots,\langle h_6\rangle, F_2,H_1,\dots,H_s, \langle k_1\rangle,\dots,\langle k_{24}\rangle,\langle k'_1\rangle ,\langle k'_2\rangle \}\]
and let $\S = \sqcup_{S \in \mathbb{S}} S$.

By Lemma \ref{lem. relative hyperbolic}, relative hyperbolicity of $F_8$ and \cite[Proposition 4.35]{dahmani2017hyperbolically}, there exists a finite set $X_2$ such that $\mathbb{S} \hookrightarrow_h (G, X_1 \cup X_2)$. By Lemma \ref{lem. change relative generating set}, we deduce:
\begin{equation}\label{eq. he for sc}
    \mathbb{S} \hookrightarrow_h (G,X_1).
\end{equation}
Let $\A \coloneqq X_1 \sqcup \S$. For some integer $\ell$ that will be specified later, consider the following word over $\S$

\[R \coloneqq h_8h^{\ell}_7\cdot\prod^{2\ell}_{t=\ell+1}\left(\prod^6_{m=1}h^\ell_m\right).\]

\begin{lemma}\label{lem. small cancellation}
    For any $\xi,\sigma>0$, there exists $L>0$ such that if $\ell>L$, the word $R$ satisfies the $W(\xi,\sigma)$ small cancellation condition with respect to the hyperbolic embedding \eqref{eq. he for sc} and the alphabet $\A$.
\end{lemma}

\begin{proof}
    Conditions (SC1) and (SC3) are obvious. By taking $\ell$ large enough, we can guarantee (SC4). For $m=1,\dots,8$, the order of $h_m$ is infinite. Let $\widehat{d}_{\langle h_m\rangle}\colon \langle h_m\rangle\times \langle h_m\rangle\rightarrow [0,\infty]$ for $m=1,\dots, 6$ and let $\widehat{d}_{F_2}\colon F_2\times F_2\rightarrow [0,\infty]$ be the corresponding relative metrics. The hyperbolic embedding \eqref{eq. he for sc} implies $\lim_{\ell\rightarrow\infty}\widehat{d}_{\langle h_m\rangle}(1,h^{\ell}_m)=\infty$ for $m=1,\dots, 6$ and $\lim_{\ell\rightarrow\infty}\widehat{d}_{F_2}(1,h_8h^\ell_7)=\infty$. Therefore, we can ensure (SC2) by taking $\ell$ large enough.
\end{proof}

\begin{lemma}\label{lem. cl}
    There exists $L>0$ such that if $\ell>L$, then $(G,\langle R \rangle, \langle R \rangle)$ is a Cohen--Lyndon triple.
\end{lemma}

\begin{proof}
    By \Cref{lem. relative hyperbolic} and \cite[Remark 4.26]{dahmani2017hyperbolically}, we have $\{ F_8 \} \hookrightarrow_h G$. So by \Cref{thm. sun cl}, for large $\ell$, the triple $(G,F_8,\ll R \rr_{F_8})$ is Cohen--Lyndon. The key property of $\langle R\rangle$ is that it is a free factor of $F_8$. Therefore the triple $(F_8,\langle R \rangle, \langle R \rangle)$ is Cohen--Lyndon. The desired result then follows from \Cref{lem. cl transitive}.
\end{proof}

Fix $r\geq 1$. Let $\xi,\sigma>0$ be the constant given by \Cref{lem. sc} with respect to $r$. Let $L_1>0$ be the constant given by \Cref{lem. small cancellation} with respect to $\xi,\sigma$. Let $L_2>0$ be the constant given by \Cref{lem. cl}. Below, we fix an $\ell>\max\{L_1,L_2\}$ and the word $R$ will correspond to this choice of $\ell$. Let
\[\bg:=G/\ll R \rr_G\]
and let $\pi\colon G\to \bg$ be the natural homomorphism. For simplicity, denote
\[\pi(\mathbb{S}):=\{\pi(S)\mid S\in\mathbb{S}\},\;\; \pi(\S) := \bigsqcup\limits_{S \in \mathbb{S}} \pi(S),  \;\; \pi(\A) := \{\pi(a)\mid a\in \A\}.\]
By \Cref{lem. sc} and our choice of $\ell$, the following hold:

\begin{enumerate}[label=(\alph*)]
    \item\label{item. s injective}$\pi$ maps $B_{\A}(r)$ injectively onto $B_{\pi(\A)}(r)$. In particular, $\pi$ restricts to an injective map on each $S\in\mathbb{S}$.
    \item\label{item. s he} $\pi(\mathbb{S})\hookrightarrow_h(\bg,\pi(X_1))$.
    \item\label{item. s centeralizer} For any $g\in G$ with $|g|_{\A}\leq r$, we have $C_{\bg}(\pi(g))=\pi(C_G(g))$.
    \item\label{item. s torsion} For each $\bar g\in \bg$ with finite order, there exists $g\in G$ with finite order such that $\bar g=\pi(g)$.
\end{enumerate}

As $\pi(X)\sqcup(\bigsqcup^s_{i=1}\pi(H_i))\subset \pi(\A)$, assertions \ref{item. injectivity} and \ref{item. centralizer} follow from items \ref{item. s injective} and \ref{item. s centeralizer}, respectively.

As $X$ is a relative generating set of $G$ with respect to $\mathbb S$, the set $\pi(X)$ is a relative generating set of $\bg$ with respect to $\pi(\mathbb S)$. By combining \ref{item. s he} with \Cref{lem. change relative generating set} we get that
\begin{equation}\label{eq. 3}
    \pi(\mathbb{S})\hookrightarrow_h(\bg, \pi(X)).
\end{equation}
Item \ref{item. s injective} implies that, for $S\in \mathbb{S} \smallsetminus\{H_i\}^s_{i=1}$, the group $\pi(S)$ is free. Assertion \ref{item. hyperbolicity} follows by combining this with \eqref{eq. 3} and \cite[Lemma 5.14]{abbott2017hyperbolic}.

Assertion \ref{item. torsion element} is exactly Item \ref{item. s torsion}. Assertion \ref{item. K absorb} follows by construction. By Item \ref{item. s he} and \Cref{thm. acylindrical action}, there exists a subset $\pi(X)\subset\overline Y\subset \bg$ such that 

\[\pi(\mathbb{S})\hookrightarrow_h(\bg,\overline Y)\]
and the action $\bg\curvearrowright \Gamma(\bg,\overline Y\sqcup \pi(\S))$ is acylindrical. As $\langle k_1 \rangle \cap \langle k_2 \rangle=\{1\}$, Item \ref{item. s injective} implies

\begin{equation}\label{eq. 4}
    \langle \pi(k_1) \rangle \cap \langle \pi(k_2) \rangle =\{1\}.
\end{equation}

Moreover, as $\langle \pi(k_1) \rangle\cup \langle \pi(k_2) \rangle\subset \pi(K)$, by combining \eqref{eq. 4} and \cite[Lemma 3.24]{chifan2023small} we get that $\pi(K)$ is suitable with respect to the action $\Gamma(\bg,\overline Y\sqcup \pi(\S))$. Similarly, from $\langle k'_1 \rangle \cap \langle k'_2 \rangle = \{1\}$ we get that $\pi(K')$ is suitable. Let 

\[\overline X:= \overline Y\sqcup \left(\bigsqcup_{S\in\mathbb S\smallsetminus\{H_i\}^s_{i=1}}\pi(S)\right).\]
Then \cite[Remark 4.26]{dahmani2017hyperbolically} implies that $\{\pi(H_i)\}^s_{i=1}\hookrightarrow_h (\bg,\overline X)$. This proves the first half of assertion \ref{item. K non-elementary}.

If $|X|<\infty$, then the action $\bg\curvearrowright \Gamma(\bg,\bx\sqcup \pi(\S))$ is already acylindrical by \Cref{thm. acylindrical action}, so we can take $\overline Y=\pi(X)$. The above argument yields that $\pi(K)$ and $\pi(K')$ are suitable with respect to the relatively hyperbolic pair $(\bg,\pi(\mathbb S))$. By \Cref{lem. suitable for rh}, each of $\pi(K)$ and $\pi(K')$ is not virtually cyclic, contains a hyperbolic element of the pair $(\bg, \pi(\mathbb S))$, and does not normalize any non-trivial finite normal subgroup of $\bg$. The same holds true even if $(\bg, \pi(\mathbb S))$ is replaced by $(\bg, \{\pi(H_i)\}^s_{i=1})$. Assertion \ref{item. K non-elementary} then follows by another application of \Cref{lem. suitable for rh}.

Finally, by Lemma \ref{lem. cl}, assertion \ref{item. gd} follows from \Cref{cor. classifying space}, and then assertion \ref{item. homology} follows from this and Lemma \ref{lem. compute cohomology}. This concludes the proof of Theorem \ref{thm. inductive step}.

\section{Relative torsion-free Tarski monsters}
\label{sec. relative}

In this section we construct relative torsion-free Tarski monsters with homological control. This proves Theorem \ref{intro:thm:relative}, but the statement below gives much more information, in a form that can be used for the applications in the next section.

\begin{theorem}[Theorem \ref{intro:thm:relative}]
\label{thm. embedding}
    Let $H_i: i\geq 1$ be countable torsion-free groups, let $L_i : i \geq 1$ be torsion-free non-elementary hyperbolic groups, and let $B_i \subset L_i: i\geq 1$ be finite subsets. Then there exists a group $M$ such that the following hold.
    \begin{enumerate}[label=(\roman*)]
        \item\label{item. simple} $M$ is a finitely generated torsion-free simple group.
        \item\label{item. Hi malnormal} $\{H_i\}_{i\geq 1}$ embeds as a malnormal family of subgroups of $M$.
        \item\label{item. tarski like} Each proper subgroup of $M$ is either cyclic or conjugate into some $H_i$.
        \item\label{item. quotients Li} There are quotient maps $ L_i \to M$ which are injective on $B_i$.
        \item\label{item. gd <} $\gd(M) \leq \max\{2, \sup_i\gd(H_i), \sup_i \gd(L_i)\}$.
        \item\label{item. homology relative} For every $M$-module $A$ and every $n \geq 3$, there are isomorphisms
        \[H_n(M;A)\cong \left( \bigoplus_{i\geq 1} H_n(H_i;A) \right) \oplus \left( \bigoplus_{i\geq 1} H_n(L_i;A) \right),\] 
        \[H^n(M;A)\cong \left( \prod_{i\geq 1} H^n(H_i;A) \right) \times \left( \prod_{i\geq 1} H^n(L_i;A) \right),\]
        as well as an injective map
        \[\left( \bigoplus_{i\geq 1} H_2(H_i;A) \right) \oplus \left( \bigoplus_{i\geq 1} H_2(L_i;A) \right) \hookrightarrow H_2(M;A),\]
        and a surjective map
        \[H^2(M; A) \twoheadrightarrow \left( \prod_{i\geq 1} H^2(H_i;A) \right) \times \left( \prod_{i\geq 1} H^2(L_i;A) \right).\]
    \end{enumerate}
\end{theorem}

This entire section is devoted to the proof of Theorem \ref{thm. embedding}. Let $G_0$ be a free group of rank $2$ on basis $X_0$, let $\phi_0$ be the identity map on $G_0$, and let $C_0=\emptyset$. For each $i\geq 1$, let $X_i$ be a finite generating set of $L_i$. Enumerate all $2$-generated subgroups of $G_0$ as $K_1, K_2, K_3, \ldots$ For each $i$, enumerate elements of $H_i$ as $g_{i,i},g_{i,i+1},g_{i,i+2},\ldots$ (if $H_i=\{1\}$, set $g_{i,j}=1$ for all $j\geq i$).  We inductively construct the following data:

\begin{itemize}
    \item a sequence of homomorphisms $ \phi_i \colon G_0\to G_i$;
    \item for $j<i$, a homomorphism $\pi_{j,i}\colon G_j\ast H_{j+1}\ast L_{j+1}\to G_i$; and
    \item a sequence of subsets $C_i\subset G_i$,
\end{itemize}
such that the following hold for all $i\geq 0$.
    \begin{enumerate}[label=(\Alph*)]
        \item\label{item. quotient}
        \parbox{\linewidth}{\centering $\pi_{i,i+1}\circ\phi_i(G_0) = \pi_{i,i+1}(L_{i+1}),$}
        \begin{equation}\label{eq. def of phi}
            \phi_{i+1}=\pi_{i,i+1}\circ \phi_i \colon G_0 \to G_{i+1},
        \end{equation}
        and
        \begin{equation}\label{eq. def of pi}
            \pi_{j,i+1}=\pi_{i,i+1}\circ \pi_{j,i}\colon G_j\ast H_{j+1}\ast L_{j+1}\to G_{i+1} \text{ for }j<i.
        \end{equation}
        
        \item\label{item. i torsion free} 
        $G_i$ is torsion-free.
        
        \item\label{item. injectivity relative} 
        $\pi_{i,i+1}$ is injective on $C_i\cup H_{i+1} \cup B_{i+1}$ and

        \begin{equation}\label{eq. def of C}
            C_{i+1} = \pi_{i,i+1}(C_i\cup H_{i+1} \cup B_{i+1}).
        \end{equation}
        
        \item\label{item. i hyperbolic} 
        $\phi_i(X_0)$ is a relative generating set of $G_i$ with respect to $\{\pi_{j,i}(H_{j+1})\}_{j<i}$ and $\{\pi_{j,i}(H_{j+1})\}_{j<i}\hookrightarrow_h (G_i,\phi_i(X_0))$. In particular, $(G_i,\{\pi_{j,i}(H_{j+1})\}_{j<i})$ is a relatively hyperbolic pair, and $\{\pi_{j,i}(H_{j+1})\}_{j<i}$ is a malnormal family of subgroups of $G_i$ by \Cref{lem. he -> malnormal}.
        
        \item\label{item. i suitable} 
        $\phi_i(G_0)$ is a suitable subgroup with respect to the relatively hyperbolic pair $(G_i,\{\pi_{j,i}(H_{j+1})\}_{j<i})$.
        
        \item\label{item. i centralizer} 
        For any $g\in \phi_i(G_0)$ with $|g|_{\phi_i(X_0)}\leq i$, it holds $C_{G_{i+1}}(\pi_{i,i+1}(g))=\pi_{i,i+1}(C_{G_i}(g))$.
        
        \item\label{item. i absorb} 
        Either $\phi_{i+1}(K_{i+1})$ is \emph{elementary} with respect to the relatively hyperbolic pair $(G_i,\{\pi_{j,i}(H_{j+1})\}_{j<i})$ (i.e., $\phi_{i+1}(K)$ is either cyclic or conjugates into some $\pi_{j,i}(H_{j+1})$), or $\phi_{i+1}(X_0)\subset \phi_{i+1}(K_{i+1})$.
        
        \item\label{item. i even} 
        For all $j < i$, we have $\pi_{j,i}(g_{j+1,i})\in \phi_i(G_0)$.
        
        \item\label{item. classifying} Given any $K(G_i\ast H_{i+1} \ast L_{i+1}, 1)$ CW-complex, one can obtain a $K(G_{i+1}, 1)$ CW-complex by attaching $2$-cells to it.
    \end{enumerate}

We proceed with the construction. Each $G_{i+1}$ will be obtained from $G_i$ in four steps. Each of these steps will involve an application of \Cref{thm. inductive step} with a sufficiently large parameter $r$. Suppose that $G_i$ has been constructed for some $i \geq 0$. For simplicity, denote
        \[\mathbb H_i:= \{\pi_{j,i}(H_{j+1})\}_{j<i}\sqcup \{H_{i+1}\},\;\; \A_i=\phi_i(X_0)\sqcup X_{i+1}\sqcup \left(\bigsqcup_{H\in \mathbb H_i}H\right).\]
We have $\{G_i,H_{i+1},L_{i+1}\}\hookrightarrow_h (G_i\ast H_{i+1}\ast L_{i+1},\emptyset)$ by \cite[Example 4.12 (c)]{dahmani2017hyperbolically}, $\{1\}\hookrightarrow_h (L_{i+1},X_{i+1})$ by definition, and $\{\pi_{j,i}(H_{j+1})\}_{j<i}\hookrightarrow_h (G_i,\phi_i(X_0))$. So \cite[Proposition 4.35]{dahmani2017hyperbolically} gives a hyperbolic embedding
\begin{equation}\label{eq. 9}
    \mathbb H_i\hookrightarrow_h(G_i\ast H_{i+1}\ast L_{i+1},\phi_i(X_0) \sqcup X_{i+1}).
\end{equation}
The action 
\begin{equation}\label{eq. 10}
    G_i \ast H_{i+1} \ast L_{i+1} \curvearrowright\Gamma\left(G_i\ast H_{i+1}\ast L_{i+1},\A_i\right)
\end{equation}
is acylindrical by \Cref{thm. acylindrical action}. By the inductive hypothesis, $\phi_i(G_0)$ is suitable with respect to the relatively hyperbolic pair $(G_i,\mathbb H_i\smallsetminus\{H_{i+1}\})$. By \Cref{lem. suitable for rh}, $\phi_i(G_0)$ is not virtually cyclic, contains a hyperbolic element with respect to $(G_i,\mathbb H_i\smallsetminus\{H_{i+1}\})$, and does not normalize any non-trivial finite normal subgroup of $G_i$. As $\phi_i(X_0)$ and $X_{i+1}$ are finite, the hyperbolic embedding \eqref{eq. 9} implies that $G_i \ast H_{i+1}\ast L_{i+1}$ is hyperbolic relative to $\mathbb H_i$. The above implies that $\phi_i(G_0)$ contains a hyperbolic element with respect to the relatively hyperbolic pair $(G_i\ast H_{i+1}\ast L_{i+1},\mathbb H_i)$ and does not normalize any non-trivial finite normal subgroup of $G_i\ast H_{i+1}\ast L_{i+1}$. Combining this with the fact that $\phi_i(G_0)$ is not virtually cyclic and using \Cref{lem. suitable for rh}, we see that $\phi_i(G_0)$ is suitable with respect to the action \eqref{eq. 10}. Similarly, $L_{i+1}$ is suitable with respect to the action \eqref{eq. 10}.

Apply \Cref{thm. inductive step} to the hyperbolic embedding \eqref{eq. 9} and the action \eqref{eq. 10}, with $K=\phi_i(G_0),K'=L_{i+1}$ and $\{g_n\}^N_{n=1}=X_{i+1}$. This results in a quotient map $\pi_{i,i+1/4}\colon G_i\ast H_{i+1}\ast L_{i+1}\to G_{i+1/4}$ such that $\pi_{i,i+1/4}(L_{i+1})\leq \pi_{i,i+1/4}\circ\phi_i(G_0)$, and $\pi_{i,i+1/4}(L_i)$ and $\pi_{i,i+1/4}\circ \phi_i(G_0)$ are suitable subgroups with respect to the action 
\begin{equation}\label{eq. 11}
    G_{i+1/4}\curvearrowright \Gamma\left(G_{i+1/4},  \pi_{i,i+1/4}(\A_i)\right).
\end{equation}

For simplicity, we will write $\pi_{i,i+1/4}(\mathbb H_i)$ for the family $\{\pi_{i,i+1/4}(H)\}_{H\in \mathbb H_i}$, and we will use similar notation later. Apply \Cref{thm. inductive step} to the hyperbolic embedding
\[\pi_{i,i+1/4}(\mathbb H_i)\hookrightarrow_h (G_{i+1/4}, \pi_{i,i+1/4}(\A_i))\]
and the action \eqref{eq. 11}, with $K=\pi_{i,i+1/4}(L_{i+1}), K'=\pi_{i,i+1/4} \circ \phi_i(G_0)$ and $\{g_n\}^N_{n=1}=\pi_{i,i+1/4}\circ \phi_i(X_0)$. This results in a quotient map $\pi_{i,i+1/2}\colon G_i\ast H_{i+1}\ast L_{i+1}\to G_{i+1/2}$, which factors through $\pi_{i,i+1/4}$, such that
\[\pi_{i,i+1/2}\circ\phi_i(G_0)=\pi_{i,i+1/2}(L_{i+1})\]
is a suitable subgroup with respect to the action
\begin{equation}\label{eq. 12}
    G_{i+1/2}\curvearrowright \Gamma(G_{i+1/2},\pi_{i,i+1/2}(\A_i)).
\end{equation}

Now, if $\pi_{i,i+1/2}\circ \phi_i(K_{i+1})$ is elementary with respect to the relatively hyperbolic pair $(G_{i+1/2},\pi_{i,i+1/2}(\mathbb H_i))$, simply let $\pi_{i,i+3/4}:=\pi_{i,i+1/2}$ and $G_{i+3/4}:=G_{i+1/2}$.

Suppose that $\pi_{i,i+1/2}\circ \phi_i(K_{i+1})$ is non-elementary. We claim that $\pi_{i,i+1/2}\circ\phi_i(K_{i+1})$ is a suitable subgroup of the relatively hyperbolic pair $(G_{i+1/2},\pi_{i,i+1/2}(\mathbb H_i))$. By assumption, $\pi_{i,i+1/2}\circ \phi_i(K_{i+1})$ is a non-trivial subgroup of the torsion-free group $G_{i+1/2}$, so has an infinite order element $k_1$. We are done if $k_1$ is hyperbolic with respect to the relatively hyperbolic pair $(G_{i+1/2},\pi_{i,i+1/2}(\mathbb H_i))$ by \Cref{lem. suitable for rh}. If $k_1$ is not hyperbolic, then there exists $k_3\in G_{i+1/2}$ and $H'_1\in \mathbb H_i$ such that $k_3k_1k_3^{-1}\in H'_1$. As $\pi_{i,i+1/2}\circ \phi_i(K_{i+1})$ does not conjugate into $H'$, it has an element $k_2$ such that $k_3k_2k_3^{-1}\not\in H'_1$. \cite[Lemma 4.4]{osin2006elementary} then implies that for some large $\ell$ the element $k_3k_2k^\ell_1k_3^{-1}$ is a hyperbolic element of the relatively hyperbolic pair $(G_{i+1/2},\pi_{i,i+1/2}(\mathbb H_i))$. Hence so is $k_2k^\ell_1\in \pi_{i,i+1/2}\circ \phi_i(K_{i+1})$. \Cref{lem. suitable for rh} then implies that $\pi_{i,i+1/2}\circ \phi_i(K_{i+1})$ is a suitable subgroup.

Apply \Cref{thm. inductive step} to the hyperbolic embedding
\begin{equation}\label{eq. 13} 
\pi_{i,i+1/2}(\mathbb H_i)\hookrightarrow_h (G_{i+1/2}, \pi_{i,i+1/2}(\A_i))
\end{equation}
and the action \eqref{eq. 12}, with $K=\pi_{i,i+1/2}\circ\phi_i(K_{i+1}), K'=\pi_{i,i+1/2}(L_{i+1})$ and $\{g_n\}^N_{n=1}=\pi_{i,i+1/2}\circ\phi_i(X_0)$. This yields a quotient map $\pi_{i,i+3/4}\colon G_i\ast H_{i+1}\ast L_{i+1}\to G_{i+3/4}$, which factors through $\pi_{i,i+1/2}$, such that $\pi_{i,i+3/4}(L_{i+1})$ is suitable with respect to the action
\begin{equation}\label{eq. 22}
    G_{i+3/4}\curvearrowright \Gamma(G_{i+3/4},\pi_{i,i+3/4}(\A_i)).
\end{equation}

Finally, apply \Cref{thm. inductive step} to the hyperbolic embedding 
\[\pi_{i,i+3/4}(\mathbb H_i)\hookrightarrow_h (G_{i+3/4}, \pi_{i,i+3/4}(\A_i))\]
and the action \eqref{eq. 22}, with $K=K'=\pi_{i+3/4}\circ\phi_i(G_0)$ and the family $\{g_n\}^N_{n=1} = \{g_{j+1,i+1}\}_{j\leq i}$. This yields a quotient $\pi_{i,i+1}\colon G_i\ast H_{i+1}\ast L_{i+1}\to G_{i+1}$. Use \eqref{eq. def of phi} (resp. \eqref{eq. def of pi}, \eqref{eq. def of C}) as the definition of $\phi_{i+1}$ (resp. $\pi_{j,i+1},C_{i+1}$). This completes the inductive construction.

\medskip

    All items except \ref{item. i hyperbolic} and \ref{item. i suitable} automatically follow from Theorem \ref{thm. inductive step}, upon choosing a sufficiently large $r$. To see \ref{item. i hyperbolic}, note that we have a hyperbolic embedding $\mathbb H_i\hookrightarrow_h(G_{i+1}, \phi_{i+1}(X_0)\sqcup \pi_{i,i+1}(X_{i+1}))$. Note also that $\phi_{i+1}(X_0)$ is a relative generating set of $G_{i+1}$ with respect to $\mathbb H_i$, as $\pi_{i,i+1}(X_{i+1})\subset \phi_{i+1}(G_0)$, and thus can be written as a product of elements in $\phi_{i+1}(X_0^{\pm 1})$. Then \ref{item. i hyperbolic} follows from \Cref{lem. change relative generating set}, because all sets involved are finite. And \ref{item. i suitable} follows because by Lemma \ref{lem. suitable for rh} the notion of a suitable subgroup of a relatively hyperbolic pair does not depend on the choice of the finite relative generating set. This completes the construction of the groups $G_i : i \geq 0$.

\medskip
    
    These groups fit into a directed system. Let $M:=\varinjlim G_i$ and let $\phi\colon G_0\rightarrow M$
    be the natural homomorphism. For each $i\geq 0$, let $\pi_i\colon H_{i+1}\ast L_{i+1}\to M$ be the natural homomorphism. Note that, by construction, $M$ is generated by $\phi(X_0)$, and so it is finitely generated. For each $i\geq 0$, as $\pi_{i,i+1}(L_{i+1})=\phi_{i+1}(G_0)$, we see that $\pi_i$ maps $L_{i+1}$ onto $M$. As $\pi_{i,i+1}$ is injective on $H_{i+1}\cup B_{i+1}$ for all $i$, the map $\pi_i$ is injective on $H_{i+1}\cup B_{i+1}$. This ensures that Item \ref{item. quotients Li} holds, and moreover that each $H_i$ embeds into $M$. Moreover $M$ is torsion-free as a limit of torsion-free groups $G_i$, which ensures that Item \ref{item. simple} holds, except for simplicity.

    Let $Z_0$ be a $K(G_0,1)$ CW-complex. For $i\geq 1$, let $Y_i$ (resp. $Y'_i$) be a $K(H_i,1)$ (resp. $K(L_i,1)$) CW-complex. From these we inductively construct a $K(G_i,1)$ CW-complex $Z_i$ for $i\geq 0$. Suppose that $Z_i$ has been constructed for some $i\geq 0$. By \ref{item. classifying}, we can construct a $K(G_{i+1},1)$ CW-complex $Z_{i+1}$ from $Z_i\vee Y_{i+1}\vee Y'_{i+1}$ by attaching $2$-cells. We have a chain of inclusions
    \[Z_0\subset Z_1\subset Z_2 \subset \cdots\]
    Let $Z:=\bigcup_{i\geq 0} Z_i$, seen as a CW-complex endowed with the weak topology. Then $\pi_1(Z)=M$. Moreover, $Z$ is aspherical as every map from a sphere to $Z$ has compact image, and therefore factors through some $Z_i$. This shows that $Z$ is a $K(M,1)$ and gives \ref{item. gd <}.

    The space $Z$ can be alternatively be constructed as follows. First take the wedge sum of $Z_0, Y_i$ and $Y'_i$ for $i\geq 1$, and then attach $2$-cells. The first step yields a $K((\Asterisk_{i\geq 1} H_i\ast L_i) \ast G_0, 1)$. Then \ref{item. homology relative} follows from \Cref{lem. compute cohomology}.

\medskip

    We have already seen that  each $\pi_i|_{H_{i+1}}$ is injective. So the following completes the proof of \ref{item. Hi malnormal}.

    \begin{lemma}
        $\{\pi_i(H_{i+1})\}_{i\geq 0}$ is a malnormal family of subgroups of $M$.
    \end{lemma}

    \begin{proof}
         Let $h\in M$ and $0\leq i\leq\ell$ such that $h\pi_i(H_{i+1})h^{-1}\cap \pi_\ell(H_{\ell+1})\neq \{1\}$, and let $k\in G_0$ be a preimage of $h$. There exists $j>\ell$ such that 
        
        \[\phi_j(k)\pi_{i,j}(H_{i+1})\phi_j(k^{-1})\cap \pi_{\ell,j}(H_{\ell+1})\neq \{1\}.\]
        
         Since $\{\pi_{p,j}(H_{p+1})\}_{p<j} < G_j$ is a malnormal family in $G_j$ by Item \ref{item. i hyperbolic}, we get $i=\ell$ and $\phi_j(k)\in\pi_{i,j}(H_{i+1})$, and so $h\in \pi_i(H_{i+1})$.
    \end{proof}

    Next, we prove \ref{item. tarski like}:

    \begin{lemma}\label{cl. cyclic}
        Each proper subgroup of $M$ is either cyclic or conjugate into some $\pi_i(H_{i+1})$.
    \end{lemma}

    \begin{proof}
        Let $H<M$ be a proper subgroup. First suppose that $H$ is abelian. Let $1\neq h\in H$ be a non-trivial element. If $h$ conjugates into some $\pi_k(H_{k+1})$, then malnormality implies that $H\leq C_M(h)$ conjugates into $\pi_k(H_{k+1})$. So let us assume $h$ does not conjugate into any of $\pi_{k,i}(H_{k+1})$. Let $i=|h|_{\phi(X_0)}$ and let $g\in G_0$ be a preimage of $h$ such that $|g|_{X_0}=i$. Consider an arbitrary element $h'\in H$ and let $g'\in G_0$ be a preimage of $h'$. In $M$, it holds $[h,h']=1$. So there exists some $j\geq i$ such that $[\phi_j(g),\phi_j(g')]=1$, i.e., $\phi_j(g')\in C_{G_j}(\phi_j(g))$. An induction using \ref{item. i centralizer} shows that $C_{G_j}(\phi_j(g))$ is the image of $C_{G_i}(\phi_i(g))$. So $h'$ belongs to the image of $C_{G_i}(\phi_i(g))$ in $M$. Therefore, $H$ is contained in the image of $C_{G_i}(\phi_i(g))$ in $M$. Note that $\phi_i(g)$ does not conjugate into any of $\pi_{k,i}(H_{k+1})$ for $k<i$. So $\phi_i(g)$ is a hyperbolic element of the torsion-free relatively hyperbolic pair $(G_i,\{\pi_{k,i}(H_{k+1})\}_{k<i})$ and therefore $C_{G_i}(\phi_i(g))$ is cyclic. Hence so is $H$.
        
        Next, suppose that $H$ is non-abelian. Then $H$ contains elements $h_1$ and $h_2$ with $[h_1,h_2]\neq 1$. Let $g_1$ (resp. $g_2$) be a preimage of $h_1$ (resp. $h_2$) in $G_0$. Let $K=\langle g_1,g_2\rangle\leq G_0$. If for all $i$, the group $\phi_i(K)$ is non-elementary for the relatively hyperbolic pair $(G_i,\{\pi_{k,i}(H_{k+1})\}_{k<i}$, then by \ref{item. i absorb}, there exists $i$ such that $\phi_i(X_0)\subset \phi_i(K)$. Then $\phi(X_0)\subset H$, and therefore $H=M$, a contradiction. 
        
        So there exists $i$ such that $\phi_i(K)$ is elementary. As $[h_1,h_2]\neq 1$, the group $\phi_i(K)$ is non-abelian, and so it conjugates into some $\pi_{k,i}(H_{k+1})$. That is, there exists $g\in G_0$ such that $\phi_i(gKg^{-1})\leq \pi_{k,i}(H_{k+1})$, and thus $\phi(gKg^{-1})\leq \pi_k(H_{k+1})$. In particular,
        \begin{equation}\label{eq. 17}
            \phi(gg_1g^{-1})\in \pi_k(H_{k+1}).
        \end{equation}

        We prove that $\phi(g)H\phi(g^{-1})\leq \pi_k(H_{k+1})$. Let $h_3$ be any non-trivial element of $H$. If $[h_1,h_3]=1$, then $[\phi(g)h_1\phi(g^{-1}),\phi(g)h_3\phi(g^{-1})]=1$. Malnormality implies that $\phi(g)h_3\phi(g^{-1})\in \pi_k(H_{k+1})$. If $[h_1,h_3]\neq 1$, let $g_3\in G_0$ be a preimage of $h_3$ and let $K'=\langle g_1,g_3 \rangle\leq G_0$. From the above argument with $h_3$ in place of $h_2$ we conclude that there exist $g'\in G_0$ and $k'\geq 0$ such that $\phi(g'K'(g')^{-1})\subset \pi_{k'}(H_{k'+1})$. That is,
        \begin{equation}\label{eq. 18}
            \phi(g'g_1(g')^{-1})\in \pi_{k'}(H_{k'+1});
        \end{equation}
        \begin{equation}\label{eq. 19}
            \phi(g'g_3(g')^{-1})\in \pi_{k'}(H_{k'+1}).
        \end{equation}

        By \eqref{eq. 18}, we have
        \[\phi((g'g^{-1})(gg_1g^{-1})(g'g^{-1})^{-1})\in \pi_{k'}(H_{k'+1}).\]
        Combining this with \eqref{eq. 17} and malnormality, we get $k=k'$ and $\phi(g'g^{-1})\in\pi_k(H_{k+1})$, which, when combined with \eqref{eq. 19}, gives $\phi(gg_3g^{-1})\in \pi_k(H_{k+1})$.
        
        Therefore, in all cases, $\phi(g)$ conjugates $h_3$ into $\pi_k(H_{k+1})$, and since $h_3$ was arbitrary this concludes the proof.
    \end{proof}  

    Finally, we prove that $M$ is simple, which gives the only remaining Item: \ref{item. simple}. Let $H\lhd M$ be a proper non-trivial normal subgroup. As each $\pi_i(H_{i+1})$ is malnormal in $M$, the group $H$ cannot conjugate into any of $\pi_i(H_{i+1})$. By \Cref{cl. cyclic}, we get that $H$ is cyclic. Let $h\in H$ be a generator, and let $g\in G_0$ be a preimage of $h$. As $H$ is normal in $M$, there exists $i$ such that $\phi_i(xgx^{-1})\in \{\phi_i(g),\phi_i(g^{-1})\}$ for all $x\in X_0$. So $\phi_i(G_0)$ is contained in $N_{G_i}(\phi_i(g))$, the normalizer of $\phi_i(g)$ in $G_i$. Note that $\phi_i(g)$ cannot conjugate into $\pi_{j,i}(H_{j+1})$ for any $j<i$, as otherwise $H$ would conjugate into $\pi_j(H_{j+1})$. As $(G_i,\{\pi_{j,i}(H_{j+1})\}_{j<i})$ is a torsion-free relatively hyperbolic pair, the element $\phi_i(g)$ is hyperbolic, and so $N_{G_i}(\phi_i(g))$ is cyclic. So $\phi_i(G_0)$ is cyclic, which contradicts that $\phi_i(G_0)$ is non-elementary. This concludes the proof of Theorem \ref{thm. embedding}.

\begin{remark}
\label{rem:fp}
    Groups constructed this way can never be finitely presented. Indeed, suppose that $\{ G_i \}_{i \geq 1}$ is a directed system where $G_i \to G_{i+1}$ is an epimorphism, and $M$ is the direct limit. If $M$ is finitely presented, then all the relations defining $M$ must appear after finitely many steps, and therefore $M = G_i$ for $i$ large enough. This is not possible if the $G_i$ are acylindrically hyperbolic, and $M$ has some property that is incompatible with acylindrical hyperbolicity, such as simplicity.
\end{remark}

\begin{remark}
    \label{rem:fp2}
    The group $M$ constructed in the proof of \Cref{thm. embedding} has infinite-dimensional $H_2(M;\Q)$, so $M$ is not even of type $FP_2(\Q)$. Keep the notation from the proof, and choose $Z_0$ (the $K(G_0, 1) = K(F_2, 1)$) to be a wedge of two circles, and $Y_i'$ (the $K(L_i, 1)$) to be a Rips complex for $L_i$, so that $Y_i'$ has exactly $|X_i|$ 1-cells.

    The quotient $G_i \ast H_{i+1} \ast L_{i+1} \to G_{i+1/4}$ introduces $|X_{i+1}|$ relations that ensure that the images of the generators of $L_{i+1}$ belong to the image of $G_0$. Then the quotient $G_{i+1/4} \to G_{i+1/2}$ introduces two relations that ensure that the images of the generators of $G_0$ belong to the image of $L_{i+1}$. So the quotient $G_i \ast H_{i+1} \ast L_{i+1} \to G_{i+1/2}$ introduces a total of $(|X_{i+1}| + 2)$ relations, that only involve $G_i \ast L_{i+1}$, and not $H_{i+1}$.

    This yields an alternative construction of $Z$ (the $K(M, 1)$), by building separately the part that involves the $L_i$ and then the part that involves the $H_i$. More precisely, we define $T_i$ inductively by setting $T_0 = Z_0$ and letting $T_{i+1}$ be the wedge $T_i \vee Y_{i+1}'$ to which we attach the $2$-cells corresponding to the relators added in the quotient $G_i \ast H_{i+1} \ast L_{i+1} \to G_{i+1/2}$. We let $T = \bigcup_{i \geq 0} T_i$ with the weak topology. The passage from $T_i$ to $T_{i+1}$ introduces $|X_{i+1}|$ $1$-cells and $(|X_{i+1}| + 2)$ $2$-cells. Computing the Mayer--Vietoris sequence of this amalgam shows that $H_2(T_i; \Q) \to H_2(T_{i+1}; \Q)$ is injective, and moreover $b_2(T_{i+1}; \Q) \geq b_2(T_i; \Q) + 2$. Since homology commutes with direct limits, we see that $H_2(T; \Q)$ is infinite-dimensional.
    
    Now take the wedge of $T$ with all of the $Y_i$, and attach the $2$-cells corresponding to the relators added in the quotient $G_{i+1/2} \to G_{i+1}$, for all $i \geq 0$. The result is $Z$. Because $Z$ is obtained from $T\vee \left(\bigvee_{i\geq 1}Y_i\right)$ by adding $2$-cells, by Lemma \ref{lem. compute cohomology} there is an injective map $H_2(T; \Q) \to H_2(Z; \Q)$. So the previous paragraph implies that $H_2(M; \Q) = H_2(Z; \Q)$ is infinite-dimensional.
\end{remark}

It is an interesting problem to find finitely presented groups that satisfy a version of \Cref{thm. embedding}, at least in special cases. Note that already in the case of torsion-free Tarski monsters, there is no known finitely presented example.

\section{Applications}
\label{sec:applications}

\subsection{Finitely generated simple groups}

The easiest consequences of Theorem \ref{thm. embedding} are Theorem \ref{intro:thm:simple:embedding} and Corollary \ref{intro:cor:simple}. Let us note that embedding into infinite-dimensional finitely generated simple groups was already achieved e.g. by Thompson \cite{embeddingsimple:thompson}, so we focus here on the finite-dimensional case.

\begin{theorem}[Theorem \ref{intro:thm:simple:embedding}]
\label{thm:simple:embedding}
    Let $G$ be a countable group of geometric dimension at most $n \geq 2$. Then there exists a finitely generated simple group $M$ with $\hd_R(M) = \cd_R(M) = \gd(M) = n$, for every commutative unital ring $R$, into which $G$ embeds as a malnormal subgroup.
\end{theorem}

\begin{proof}
    Let $H_1=G$, which is torsion-free as having finite geometric dimension; let $L_1$ be the fundamental group of a closed orientable hyperbolic $n$-manifold; let $H_i=\{1\}$ for $i\geq 2$; and let $L_i : i \geq 2$ be free groups. We can apply Theorem \ref{thm. embedding} to obtain a finitely generated simple group $M$ into which $G$ embeds malnormally. It has geometric dimension at most $n$. Moreover $H_n(M; R) \supset H_n(L; R) \neq 0$ and therefore $n \leq \hd_R(M) \leq \cd_R(M) \leq \gd(M) \leq n$.
\end{proof}

\begin{corollary}[Corollary \ref{intro:cor:simple}]
\label{cor:simple}
    For all $n \geq 3$, there exist continuum many pairwise non-isomorphic finitely generated simple (torsion-free) groups $G$ with $\hd_R(G) = \cd_R(G) = \gd(G) = n$.
\end{corollary}

Note that the case of $n = 2$ was already covered by Camm \cite{camm}.

\begin{proof}
    Let $\{G_i\}_{i \in I}$ be a family of continuum many pairwise non-isomorphic finitely generated groups with $\gd(G_i) = n$. For example, we can start with $n = 2$ \cite{camm} and then take free products with the fundamental group of a closed aspherical $n$-manifold. Theorem \ref{thm:simple:embedding} produces a family $\{ M_i \}_{i \in I}$ of finitely generated simple torsion-free groups of the right dimension, such that $M_i$ contains $G_i$. Since every countable group contains only countably many finitely generated subgroups, there must be continuum many pairwise non-isomorphic groups in $\{ M_i \}_{i \in I}$.
\end{proof}

\subsection{Finite-dimensional torsion-free Tarski monsters}

Theorem \ref{thm. embedding}, applied with $L_1$ the fundamental group of a closed oriented hyperbolic $n$-manifold easily gives torsion-free Tarski monsters of dimension $n$. Here we use the additional control to produce infinitely many, for $n \geq 3$, and continuum many, for $n \geq 4$, proving Theorem \ref{intro:thm:fd:gap}.

\begin{corollary}[Theorem \ref{intro:thm:fd:gap}, first part]
\label{cor. 3 tarski}
    For all $n \geq 3$, there exist infinitely many pairwise non-isomorphic torsion-free Tarski monsters $M$ such that $\hd_R(M) = \cd_R(M) = \gd(M) = n$ for all commutative unital rings $R$.
\end{corollary}

\begin{proof}
    We let $H_i = \{ 1 \}$ for $i\geq 1$, let $ L_i$ be the fundamental group of a closed oriented hyperbolic $n$-manifold for $1\leq i\leq k$, and let $L_i$ be a free group for $i>k$. Theorem \ref{thm. embedding} then produces a torsion-free Tarski monster $M$ such that $\gd(M) \leq n$ and
    \[H_n(M;R) \cong \bigoplus_{i = 1}^k H_n(L_i;R) \cong R^k.\]
    This gives $H_n(M;R) \neq 0$ and thus $n \leq \hd_R(M) \leq \cd_R(M) \leq \gd(M) \leq n$. Varying $k$, we obtain infinitely many examples.
\end{proof}

It is clear that such an argument cannot produce continuum many groups with different homology. We achieve this for $n \geq 4$, by using torsion in $H^3$. The main ingredient for the construction is the following.

\begin{lemma}\label{lem. torsion building block}
    For each prime $p$, there exists a hyperbolic group $G$ with $\gd(G) \leq 4$ such that $\tor H^3(G;\Z)$ is a direct sum of $\Z/p$.
\end{lemma}

\begin{proof}
    The desired groups will be constructed via two Dehn fillings. The first one is a variant of the Dehn filling in \cite{ratcliffe2005some}. By \cite[Theorem 4.3]{ivansic2004hyperbolic}, there is a link $L$ in $S^4$ with five components, each of which is homeomorphic to $S^1\times S^1$, such that $S^4 - L$ is an orientable cusped hyperbolic manifold. The link $L$ has a closed tubular neighborhood $N(L)$ in $S^4$, which has five connected components. Each of these connected components is homeomorphic to $S^1\times S^1\times D^2$. Let $Y$ be the complement of the interior of $N(L)$ in $S^4$. Then $\partial Y$ has five connected components, each of which is homeomorphic to $S^1\times S^1\times S^1$.
    
    By the Mayer--Vietoris sequence, we have
\[0\to \bigoplus^5_{i=1} H_1(S^1\times S^1\times S^1;\Z) \to \bigoplus^5_{i=1} H_1(S^1\times S^1\times D^2;\Z)\oplus H_1(Y;\Z)\to 0.\]

So $H_1(Y;\Z)\cong \Z^5$. For $i=1,\dots,5$, let $\{a_i,b_i,c_i\}$ be a basis of the fundamental group of the $i$-th $3$-torus, such that $c_i$ is the boundary of the $D^2$ factor of the $i$-th solid $4$-torus. Below, we will abuse notation and use $a_i,b_i,c_i$ to represent their images in other groups, such as $\pi_1(Y)$ and $H_1(Y;\Z)$. Then $\{c_1,\dots,c_5\}$ is a basis of $H_1(Y;\Z)$.

Fix a prime $p$. Consider a large prime $q\neq p$. Glue, for each $i$, a solid $4$-torus $S^1\times S^1\times D^2$ to the $i$-th $3$-torus boundary of $Y$ such that the boundary of $D^2$-factor of the solid torus is mapped to $qa_i+pc_i$. By \cite[Theorem 2.7]{fujmann2010}, as long as $q$ is large enough, the resulting space will be a manifold $X$ with a locally CAT(0) metric. In particular, $X$ is aspherical. Moreover, we have 
\[H_1(X;\Z)=H_1(Y;\Z)/\langle qa_i+pc_i,\;\; i=1,\dots,5 \rangle \cong (\Z/p)^5,\]
where the last equality follows from the fact that, for $i=1,\dots,5$, we have $a_i=0$ in $H_1(Y;\Z)$. The manifold $X$ is closed and orientable, as it is obtained by gluing two orientable manifolds along boundaries. Therefore, Poincar\'e duality gives $H^3(X;\Z) \cong H_1(X;\Z)\cong (\Z/p)^5$. The manifold $X$ has the same Euler characteristic as $S^4$, and from the above computation we have $b_1(X)=b_3(X)=0$ and $b_4(X)=1$. So $b_2(X)=0$. As $X$ is aspherical, we have $H^3(\pi_1(X);\Z)\cong (\Z/p)^5$ and $b_2(\pi_1(X))=0$.

By \Cref{thm. osin df}, as long as $q$ is large, the group $\pi_1(X)$ will be hyperbolic relative to five copies of $\Z^2$, one for each component of $L$. A basis for the $i$-th copy of these $\Z^2$ is given by $\{b_i, rc_i-sa_i\}$, where $r,s$ are integers such that $rp-qs=1$. Consider large coprime integers $m,n$ and the group

\[G:=\pi_1(X)/\ll mb_i+n(rc_i-sa_i),\;\; i=1,\dots,5 \rr_{\pi_1(X)}\]

By \Cref{cor. osin df} and \cite[Theorem 4.10]{sun2019cohomologyii}, as long as $m,n$ are large enough, the group $G$ will be hyperbolic and satisfy $\cd_{\Z}(G)\leq 4$, and therefore $\gd(G) \leq 4$ \cite{eilenbergganea}. Moreover there will be a spectral sequence

\[E^{p,q}_2=
\begin{cases}
    H^p(\Z;H^q(\Z;\Z)),& q\geq 1\\
    H^p(G;\Z),& q=0
\end{cases}
\Rightarrow H^{p+q}(\pi_1(X);\Z).
\]

There is a differential $d\colon H^1(\Z;H^1(\Z;\Z))\to H^3(G;\Z)$. Upon inspection of $E^{p,q}_2$, we see that $\coker d=H^3(\pi_1(X);\Z)=(\Z/p)^5$. We have $H^1(\Z;H^1(\Z;\Z))=\Z$ because the inside $H^1(\Z;\Z)$ is computed with respect to the trivial action, and $\Z$ has the trivial action on $H^1(\Z;\Z)$ as $\Z^2$ is abelian. Combining this with $b_2(X)=0$, we get that $\ker d=0$. So $\tor H^3(G;\Z)$ is a direct sum of $\Z/p$.
\end{proof}

\begin{theorem}[Theorem \ref{intro:thm:fd:gap}, second part]
\label{thm:fd:continuum}
    For each integer $n\geq 4$, there exist continuum many pairwise non-isomorphic torsion-free Tarski monsters $M$ such that $\hd_R(M)=\cd_R(M)=\gd(M)=n$ for all commutative unital rings $R$.
\end{theorem}

\begin{proof}
    Let $L_1$ be the fundamental group of an orientable hyperbolic $n$-manifold. Since $H^3(L_1;\Z)$ is a finitely generated abelian group, there exists an integer $p_1$ such that $H^3(L_1;\Z)$ does not have $p$-torsion for all $p > p_1$. Let $p_1<p_2<\dots$ be a sequence of primes. For each $i\geq 2$, there is a hyperbolic group $L_i$ with $\gd(L_i)\leq 4$ and $\tor H^3(L_i;\Z)$ being a direct sum of $\Z/p_i$, by \Cref{lem. torsion building block}.

    For each infinite subset $1\in S\subset \N$ containing $1$, apply \Cref{thm. embedding} to the sequences $\{H_i=\{1\}\}_{i\in S}$ and $\{L_i\}_{i\in S}$. The resulting torsion-free Tarski monster $M_S$ satisfies that $H^3(M_S;\Z)$ contains $p$-torsion for some $p>p_1$ if and only if $p\in S$, which distinguishes the $M_S$ up to isomorphism. Moreover    
    \[\gd(M_S)\leq \max\{2,n,4\}=n,\]
    \[H_n(M_S;R)=\bigoplus_{i \in S} H_n(L_i;R) \neq 0.\] 
    Therefore $n \leq \hd_R(M)\leq\cd_R(M)\leq\gd(M) \leq n$.
\end{proof}

\begin{remark}
\label{rem:3dim}
    It would be good to extend Theorem \ref{thm:fd:continuum} to dimension $3$. This would need an analog of Lemma \ref{lem. torsion building block} for $3$-dimensional groups: it suffices to construct a family of $3$-dimensional hyperbolic groups with prescribed torsion in $H^3$ (by the Universal Coefficient Theorem, this is equivalent to constructing a family of $3$-dimensional hyperbolic groups with prescribed torsion in $H_2$).

    One promising approach would be to construct a $3$-dimensional hyperbolic group $A$ with a $2$-dimensional subgroup $B$ such that $H_2(A;\Z) = 0$ and $\ker(H_1(B;\Z) \to H_1(A; \Z))$ contains $p$-torsion. Then the double $G = A \ast_{B} A$ is $3$-dimensional and contains $p$-torsion in $H^3$. If moreover $B$ is quasiconvex and malnormal, $G$ is hyperbolic \cite{bestvina1992combination}.
    Elaborations of the Kahn--Markovic construction \cite{kahn:markovic} allow for such examples \cite{vhtorsion1, vhtorsion2}, except, crucially, for malnormality, which is necessary for the double to be hyperbolic.
\end{remark}

\subsection{Infinite-dimensional torsion-free Tarski monsters \\ and measure equivalence}

Theorem \ref{thm. embedding}, applied with $L_i$ fundamental groups of closed oriented hyperbolic manifolds of larger and larger dimension, easily gives an infinite-dimensional torsion-free Tarski monster. However, to distinguish these analogously to Theorem \ref{thm:fd:continuum} we would need control the $n$-th homology of \emph{all} the $L_i$, for some fixed $n \geq 3$. Thanks to partial results on the Singer Conjecture, it is easier to keep track of $L^2$-Betti numbers, and this will also allow to distinguish the torsion-free Tarski monsters up to measure equivalence.

\begin{theorem}[Theorems \ref{intro:thm:jump} and \ref{intro:thm:ME}]
\label{thm. infdim tarski}
    There exist continuum many pairwise non-measure equivalent finitely generated torsion-free Tarski monsters $M$ such that $\hd_R(M) = \cd_R(M) = \gd(M) = \infty$ for all commutative unital rings $R$.
\end{theorem}

\begin{proof}
    For each integer $i\geq 1$, let $L_i$ be a cocompact arithmetic lattice of $SO(2i,1)$ of the simplests type. By Selberg's lemma we may assume that $L_i$ is torsion-free. Then $L_i$ is the fundamental group of a hyperbolic $2i$-manifold. By \cite[Theorem 2.3]{jost2000vanishing}, $b^{(2)}_n(L_i)=0$ for all $n\neq i$ and $b^{(2)}_i(L_i)=\chi(L_i)$, which is non-zero by the Chern--Gauss--Bonnet formula. Note that each $L_i$ is residually finite as it is linear. Using \Cref{cor. fi l2 betti} and passing to a deep enough finite index subgroup of $L_i$, we may assume that $b^{(2)}_i(L_i)>3$. By \cite[Theorem 1.8]{bergeron2011hyperplane}, the group $L_i$ is virtually compact special, and thus is virtually locally indicable (see e.g., \cite[Proposition 4.12 (3)]{petrosyan2024l2}). For each $i$, let $B'_i\subset L_i\smallsetminus\{1\}$ be the finite set given by \Cref{thm. l2 betti approximation} with respect to the group $L_i$ and the constants $k = 2i$ and $\delta=2^{-i}$. Let also $B_i=B'_i\cup \{1\}$.

    For each subset $S\subset \N_{\geq 3}$ that contains $3$, apply \Cref{thm. embedding} with the sequences $\{H_i=\{1\}\}_{i\in S}$ and $\{L_i\}_{i\in S}$ and the finite sets $\{B_i\}_{i\in S}$. The resulting group $M_S$ is a common quotient of $\{L_i\}_{i\in S}$ such that for each $i\in S$, the quotient map $L_i\to M_S$ is injective on $B_i$. As $B_i=B'_i\cup \{1\}$, the kernel $\ker(L_i\to M_S)$ has trivial intersection with $B'_i$. The group $M_S$ is also a torsion-free Tarski monster and satisfies

    \[H_n(M_S; \mathcal{N}(M_S))\cong \bigoplus_{i\in S}H_n(L_i;\mathcal{N}(M_S)),\;\; n\geq 3.\]
    For all $n \in S$, we have

    \begin{align*}
        \left|b^{(2)}_n(M_S)-b^{(2)}_n(L_n)\right| &= \left| \left( \sum_{i\in S} b^{(2)}_n(L_i;M_S) \right) -b^{(2)}_n(L_n)\right|\\
        &= \left|\sum_{i\in S}b^{(2)}_n(L_i;M_S)-\sum_{i\in S} b^{(2)}_n(L_i)\right|\\
        &\leq \sum_{i\in S} \left|b^{(2)}_n(L_i;M_S)- b^{(2)}_n(L_i)\right|<1.
    \end{align*}
    To see the last inequality, we note that by the choice of $B_i$, for all $n \leq 2i$ it holds $|b^{(2)}_n(L_i;M_S)- b^{(2)}_n(L_i)| < 2^{-i}$; while for $n > 2i$ it holds $b^{(2)}_n(L_i;M_S) = b^{(2)}_n(L_i) = 0$, because $L_i$ is $2i$-dimensional. Therefore
    \begin{equation}\label{eq. 5}
        2<b^{(2)}_n(L_n)-1<b^{(2)}_n(M_S)<b^{(2)}_n(L_n)+1,\;\; \text{for } n\in S.
    \end{equation}
    Similarly, for $n \notin S$, we have
    \begin{equation}\label{eq. 6}
        b^{(2)}_n(M_S) = \sum_{i\in S} b^{(2)}_n(L_i;M_S)<1.
    \end{equation}

    Consider two subsets $S\neq S'\subset \N_{\geq 3}$, both of which contain $3$. Assume that $M_S$ is measure equivalent to $M_{S'}$. Then for all $n$ we have $b^{(2)}_n(M_S)/b^{(2)}_3(M_S)=b^{(2)}_n(M_{S'})/b^{(2)}_3(M_{S'})$ by \Cref{thm. gaboriau} (note that as both $S$ and $S'$ contain $3$, these ratios are well-defined). Let $n\in S\Delta S' \neq \emptyset$. Then by \eqref{eq. 5} and \eqref{eq. 6}, one of $b^{(2)}_n(M_S)/b^{(2)}_3(M_S)$ and $b^{(2)}_n(M_{S'})/b^{(2)}_3(M_{S'})$ is larger than $2/(b^{(2)}_3(L_3)+1)$ and the other one is smaller than $1/(b^{(2)}_3(L_3)-1)$. So
    \[2/(b^{(2)}_3(L_3)+1)<1/(b^{(2)}_3(L_3)-1);\]
    which is impossible since $b^{(2)}_3(L_3)>3$.
\end{proof}

\begin{corollary}[Corollary \ref{intro:cor:smith}]
\label{cor:smith}
    Let $M$ be a group as in Theorem \ref{thm. infdim tarski}. Then every admissible action of $M$ on a finite-dimensional contractible CW-complex has a global fixed point.
\end{corollary}

The proof will use Kropholler's hierarchy $\HF$, we refer the reader to Subsection \ref{ss. hierarchy} where this is treated in detail.

\begin{proof}
    Suppose by contradiction that $M$ has an admissible action on a finite-dimensional contractible CW-complex without a global fixed point. Then every stabilizer has to be either trivial or isomorphic to $\Z \in \HF$. By definition, this implies that $M \in \HF$. But a theorem of Petrosyan \cite[Theorem 3.2]{petrosyan:jump1} states that a torsion-free group in $\HF$ cannot have a jump in the cohomological dimension of subgroups, and we reach a contradiction.
\end{proof}

\subsection{Dimension spectra}

Now we prove Theorem \ref{intro:thm:spectrum}. Let us make the definition from Question \ref{q:spectrum} more precise.

\begin{definition}
\label{def. spectrum}
    For a group $G$, let
    \[S_{\gd}(G)=\{\gd(H)\mid H\leq G\}.\]
    Define similarly $S_{\hd_R}$ and $S_{\cd_R}$ for a commutative unital ring $R$.
\end{definition}

\begin{definition}
\label{def. realization}
    We say that a subset $S\subset \N\cup\{\infty\}$ is \emph{realized} by a group $G$ if $S=S_{\hd_R}(G) = S_{\cd_R}(G) = S_{\gd}(G)$ for every commutative unital ring $R$. We say that it is \emph{sharply realized} by $G$ if moreover $\hd_R(G) = \cd_R(G) = \gd(G)$ is only attained by $G$ itself.
\end{definition}

Let us start by excluding some basic cases:

\begin{lemma}
\label{lem. jump exceptions}
    Suppose that $G$ realizes $S$. Then exactly one of the following holds:
    \begin{itemize}
        \item $S = \{ 0 \}$, equivalently $G$ is trivial;
        \item $S = \{ 0, \infty \}$, in which case $G$ is an infinite torsion group;
        \item $S = \{0, 1\}$, equivalently $G$ is a non-trivial free group;
        \item $S$ is a finite set properly containing $\{0, 1\}$;
        \item $S$ is an infinite set containing $\{0, 1, \infty\}$.
    \end{itemize}
\end{lemma}

\begin{proof}
    Clearly $S = \{0\}$ if and only if $G$ is trivial, and every $S$ must contain $0$. Since groups of finite cohomological dimension are torsion-free, as soon as some $n \in \N \setminus \{0 \}$ belongs to $S$, also $1 \in S$ as $G$ contains an infinite cyclic group. So if $S = \{0, \infty\}$ then $G$ must be a non-trivial torsion group. If $G$ is non-trivial and finite, then the different spectra do no match, indeed $S_{\cd_{\Z}} = \{0, \infty\}$ but $S_{\cd_{\Q}} = \{ 0 \}$, which does not fit our definition of realizability. Finally $S = \{0, 1 \}$ if and only if $G$ is free (this is clear for $S_{\gd}$, and our definition demands in particular that $S = S_{\gd}(G)$).

    Now suppose that $S$ is not of the form above. Then $\{0, 1\} \subsetneq S$, and if $S$ is finite then we are in the fourth case. If $S$ is infinite, then $G$ itself must be infinite-dimensional and thus $\infty \in S$.
\end{proof}

\begin{remark}
\label{rem:realize:torsion}
    In our definition of realization, we require that the dimension spectra coincide for all notions of dimension, so for example in the proof of Lemma \ref{lem. jump exceptions} we saw that non-trivial finite groups do not realize a single dimension spectrum. This makes realizing $\{ 0, \infty \}$ a non-trivial task. One could approach this analogously to Theorem \ref{thm. infdim tarski}, building an infinite group $G$ such that every proper subgroup is finite (in the spirit of \cite{tarskip, rf:monsters}) but such that $H_n(G; R) \neq 0$ for infinitely many $n$, and all commutative unital rings $R$. Such a method could also be used to build groups $G_n$ such that $S_{\cd_{\Q}}(G_n) = \{0, n\}$, for all $n \geq 2$.

    An example in this direction is Grigorchuk's group $G$ \cite{grigorchuk:group}. Since this is an infinite $2$-group, it is easy to see that $S_{\hd_R}(G) = S_{\cd_R}(G) = S_{\gd}(G) = \{0, \infty\}$ whenever $2$ is not invertible in $R$. Suppose instead that $2$ is invertible in $R$. Finite subgroups of $G$ will have $\hd_R = \cd_R = 0$. The same argument as in \cite[Section 4]{branch1} shows that every finitely generated infinite subgroup of $G$ has $\hd_R = \cd_R = \infty$. However, $G$ contains infinite locally finite subgroups \cite{rozhkov}. Every such subgroup will have $\hd_R = 0$, but $\cd_R = 1$. So again this group does not realize a single dimension spectrum.
    
    In Definition \ref{def. realization}, one could strengthen the requirement on $G$ as follows: let us say that $G$ \emph{exactly realizes $S$} if $\hd_R(H) = \cd_R(H) = \gd(H)$ for all $H \leq G$ and one of the dimension spectra (and therefore every dimension spectrum) is equal to $S$. Theorem \ref{thm. realization} constructs groups with this stronger property. However, $\{0, \infty\}$ cannot be exactly realized in this sense, because groups with $S_{\cd_{\Z}} = \{0, \infty\}$ must contain finite subgroups, which have $\cd_{\Z} = \infty$ and $\cd_{\Q} = 0$.
\end{remark}

Note that $S = \{ 0, 1 \}$ is not realized by a simple group, by the above, and similarly for $S = \{0\}$ - depending on whether or not one considers the trivial group to be simple. $S = \{0, \infty\}$ is discussed in Remark \ref{rem:realize:torsion}. We show that the remaining two cases are realized by finitely generated simple torsion-free groups.

\begin{theorem}[Theorem \ref{intro:thm:spectrum}]
\label{thm. realization}

Let $S\subset \N\cup\{\infty\}$ satisfy either one of the following:
    ~
    \begin{enumerate}[label=(\roman*)]
    \item\label{item. finite} $S$ is finite and properly contains $\{0, 1\}$;
    \item\label{item. infinite} $S$ is infinite and contains $\{0, 1, \infty\}$.
    \end{enumerate}
Then there exists a finitely generated torsion-free simple group that exactly sharply realizes $S$.
\end{theorem}

\begin{proof}
\emph{(i)} First assume $|S|=3$. If $\sup S=2$, then the result is covered by Ol'shanski\u{\i}'s torsion-free Tarski monsters \cite{noetheriangroup,ol2012geometry}. If $\sup S\geq 3$, then the result follows from \Cref{cor. 3 tarski}.
Let us assume $|S|\geq 4$. For each $i\in S$ with $1<i<\sup S$, by the above paragraph there exists a finitely generated torsion-free group $H_i$ that exactly realizes $\{0,1,i\}$. For all other $i\in \N$, let $H_i=\{1\}$.

Suppose first that $n \coloneqq \sup S<\infty$, and note that $n \geq 3$ because $|S| \geq 4$. Let $L_1$ be the fundamental group of a closed orientable hyperbolic $n$-manifold, and for $i\geq 2$ let $L_i$ be the free group on two generators. \Cref{thm. embedding}, applied to this data, yields a finitely generated torsion-free simple group $M$ such that

\begin{itemize}
    \item $\gd(M)\leq n$,
    \item $H_n(M;R)= H_n(L_1;R) = R$ for all commutative unital rings $R$, 
    \item each $H_i$ embeds as a subgroup of $M$, and
    \item each proper subgroup of $M$ is either cyclic or conjugate into some $H_i$.
\end{itemize}
These properties ensure that $M$ exactly sharply realizes $S$.

Now suppose that $\sup S=\infty$. Let $L_1$ be a non-abelian free group, and for $i\geq 2$ let $L_i$ be the fundamental group of a closed orientable hyperbolic $i$-manifold. \Cref{thm. embedding}, applied to this data, yields a finitely generated torsion-free simple group $M$ such that

\begin{itemize}
    \item $R\cong H_i(L_i;R)\hookrightarrow H_i(M;R)$ for all $i\geq 2$ and all commutative unital rings $R$,
    \item each $H_i$ embeds as a subgroup of $M$, and
    \item each proper subgroup of $M$ is either cyclic or conjugate into some $H_i$.
\end{itemize}
These properties ensure that $M$ exactly sharply realizes $S$.

\medskip

\emph{(ii)} For each $i\in S$ with $1<i< \infty$, by the above there exists a finitely generated torsion-free group $H_i$ that exactly realizes $\{0,1,i\}$. For all other $i\in \N$, let $H_i=\{1\}$. For all $i\geq 1$, let $L_i$ be the free group on two generators. \Cref{thm. embedding}, applied to this data, yields a finitely generated torsion-free simple group $M$ such that

\begin{itemize}
    \item each $H_i$ embeds as a subgroup of $M$, and
    \item each proper subgroup of $M$ is either cyclic or conjugate into some $H_i$.
\end{itemize}

For all commutative unital rings $R$, as $\hd_R(M) \geq \hd_R(H_i)=i$ for all $i\in S\smallsetminus\{0,1,\infty\}$, we have $\hd_R(M)=\infty$, whence $\hd_R(M)=\cd_R(M)=\gd(M)=\infty$. The above two properties ensure that $M$ exactly sharply realizes $S$.
\end{proof}

\subsection{Kropholler's hierarchy}
\label{ss. hierarchy}

Let $\X$ be a class of groups. We define by transfinite induction a hierarchy of classes of groups. We set $\HX[0] \coloneqq \X$, and for every ordinal $\alpha$:
\begin{itemize}
    \item If $\alpha$ is a successor ordinal, we let $\HX[\alpha]$ be the class of groups $G$ that have an admissible action on a finite dimensional contractible CW-complex with stabilizers in $\HX[\alpha - 1]$.
    \item If $\alpha$ is a limit ordinal, then we set $\HX[\alpha] = \bigcup\limits_{\beta < \alpha} \HX[\beta]$.
\end{itemize}
We write $\HX$ for the union of all $\HX[\alpha]$. From now on we will assume that $\X$ is subgroup-closed. In particular $\X$ contains the trivial group. As usual, we assume that if a group is in $\X$ then every group isomorphic to it is also in $\X$.

Recall that given a group $G$ and a collection of subgroups $\{K_j\}_{j \in J}$, we say that the \emph{cohomological dimension} $\cd(G, \{K_j\}_{j \in J}) \leq N$ if the restriction
\[H^n(G; A) \to \prod\limits_{j \in J} H^n(K_j; A)\]
is an isomorphism for all $n > N$ and an epimorphism for $n = N$, all $G$-modules $A$.

\begin{lemma}
\label{lem relative cd}
    Suppose that $\cd(G, \{K_j\}_{j \in J}) < \infty$, and $K_j \in \HX[\alpha]$ for all $j \in J$. Then $G \in \HX[\alpha+1]$.
\end{lemma}

\begin{proof}
    This follows from Alonso's relative version of the Eilenberg--Ganea Theorem \cite[Theorem 3]{alonso:relative}. The statement gives an \emph{acyclic} complex, but this can be replaced by a \emph{contractible} complex in this case, see \cite[Remark (2) after Theorem 3 in Section 4]{alonso:relative}.
\end{proof}

This allows to prove that our relative Tarski monster construction, under certain conditions, preserves the class $\HX$, with control on the complexity.

\begin{corollary}
\label{cor monsters HX}
    With the notation of Theorem \ref{thm. embedding}, suppose that $\sup_i \cd(L_i) < \infty$, and that $H_i \in \HX[\alpha]$ for all $i \geq 1$. Then $M \in \HF[\alpha+1]$.
\end{corollary}

\begin{proof}
    Theorem \ref{thm. embedding}\ref{item. homology relative} gives an isomorphism
    \[H^n(M;A) \cong \prod\limits_{i \geq 1} H^n(H_i; A)\]
    for all $n > \sup_i \cd(L_i)$ and all $M$-modules $A$. By naturality, this isomorphism is induced by the restriction, and so $\cd(M; \{H_i\}_{i \geq 1}) < \infty$. We conclude by Lemma \ref{lem relative cd}.
\end{proof}

The rest of this subsection is devoted to the proof of Theorem \ref{intro:thm:hierarchy}, which we recall for the reader's convenience.

\begin{theorem}[Theorem \ref{intro:thm:hierarchy}]
\label{thm. hierarchy}
    Let $\X$ be a subgroup-closed class of groups. Suppose that there exists a countable torsion-free group in $\HX[1] \minus \X$. Then, for every countable ordinal $\alpha\geq 1$, there exists a finitely generated torsion-free simple group in $\HX[\alpha+1] \minus \HX[\alpha]$.
\end{theorem}

For $j \geq 1$, let $G_j$ be the fundamental group of a closed orientable hyperbolic $(j+1)$-manifold. For each countable ordinal $\alpha \geq 1$, consider the following condition on a countable group $M_\alpha$.

\begin{enumerate}[label=($\ast$)]
    \item\label{item. cohomology hierarchy} $M_\alpha$ has two malnormal families of subgroups $\{M_{\alpha, j}\}_{j\geq 1}$ and $\{M'_{\alpha,j}\}_{j\geq 1}$ such that the following hold:
    
    \begin{enumerate}[label=(\roman*)]
        \item For all $j$, it holds that $M'_{\alpha,j}\in \HX[\alpha]$.
        \item\label{item. limit ordinal} If $\alpha$ is a limit ordinal, then $M_{\alpha, j}\in \HX[\beta_j + 1] \minus \HX[\beta_j]$ where $\sup\{ \beta_j : j \geq 1 \} = \alpha$.
        \item\label{item. non-limit ordinal} If $\alpha$ is a successor ordinal, then either $M_{\alpha,j}\in \HX[\alpha]\minus\HX[\alpha-1]$ for all $j$, or $M_{\alpha,j_0}\in \HX[\alpha+1]\minus \HX[\alpha]$ for some $j_0\geq 1$ and for all $j\neq j_0$ it holds that $M_{\alpha,j}\in \HX[1]$.
        \item For each $j\geq 1$, there is a surjection $G_j\to M_{\alpha,j}$ and an inclusion $M'_{\alpha,j}\to M_{\alpha,j}$.
        \item\label{item. monster hierarchy} Each proper subgroup of $M_\alpha$ is either isomorphic to $\Z$ or conjugate into some $M_{\alpha,j}$. For each $j\geq 1$, each proper subgroup of $M_{\alpha,j}$ is either isomorphic to $\Z$ or conjugate into $M'_{\alpha,j}$.
        \item\label{item. detail cohomology hierarchy} For every $M_\alpha$-module $A$, $n \geq 3$ and $j\geq 1$, the above inclusions and surjections induce isomorphisms
        \begin{align*}
            &H^n(M_\alpha; A) \cong \prod_{i \geq 1} H^n(M_{\alpha, j}; A),\\
            &H_n(M_\alpha; A) \cong \bigoplus_{j \geq 1} H_n(M_{\alpha, j}; A),\\
            &H^n(M_{\alpha, j}; A)\cong H^n(M'_{\alpha,j};A)\times H^n(G_j;A),\\
            &H_n(M_{\alpha, j}; A)\cong H_n(M'_{\alpha,j};A)\times H_n(G_j;A).
        \end{align*}
    \end{enumerate}
\end{enumerate}

Our first goal is to show that this more technical statement implies what we want.

\begin{lemma}
\label{lem cohomology implies hierarchy}
    Let $\alpha \geq 1$ be a countable ordinal, and suppose that $M_\alpha$ is a countable group satisfying \ref{item. cohomology hierarchy}. Then $M_\alpha \in \HX[\alpha+1] \minus \HX[\alpha]$.
\end{lemma}

\begin{proof}
Suppose first that $\alpha$ is successor ordinal and satisfies the second half of \cref{item. non-limit ordinal}. Then, for all $n\geq 3$ and all $M_\alpha$-modules $A$, \cref{item. detail cohomology hierarchy} gives a natural isomorphism
\[H^n(M_\alpha;A)\cong \left(\prod_{j\neq j_0}H^n(M_{\alpha,j};A)\right)\times H^n(M'_{\alpha,j_0};A)\times H^n(G_{j_0};A).\]
\Cref{lem relative cd}; applied with $G=M_\alpha, K_{j_0}=M'_{\alpha,j_0}$ and $K_j=M_{\alpha,i}:j\neq j_0$; yields $M_\alpha\in \HX[\alpha+1]$. With a subgroup $M_{\alpha,j_0}\in \HX[\alpha+1]\minus \HX[\alpha]$, the group $M_\alpha$ cannot belong to $\HX[\alpha]$.

Now suppose that $\alpha$ is either a limit ordinal or a successor ordinal that satisfies the first half of \cref{item. non-limit ordinal}. By applying \Cref{lem relative cd}; with $G=M_\alpha$ and $K_j=M_{\alpha,j}:j\geq 1$; and using \cref{item. detail cohomology hierarchy}, we obtain $M_\alpha\in \HX[\alpha+1]$.

So it remains to show that $M_\alpha\not\in \HX[\alpha]$. Suppose, to the contrary, that there exists a contractible $M_\alpha$-CW-complex $X$ witnessing that $M_\alpha\in \HX[\alpha]$. Let $\beta<\alpha$ be an ordinal such that all isotropy groups of $M_\alpha \curvearrowright X$ lie in $\HX[\beta]$. For each $j \geq 1$, let 

\[
T_j:=
\begin{cases}
    M_{\alpha,j},&\text{if } M_{\alpha,j}\in \HX[\beta]\\
    M'_{\alpha,j},&\text{otherwise}
\end{cases}
.
\]

If $\alpha$ is a successor ordinal that satisfies the first half of \cref{item. non-limit ordinal}, then $T_j=M'_{\alpha,j}$ for all $j$. On the other hand, if $\alpha$ is a limit ordinal, then the set $\{j \geq 1 \mid T_j=M'_{\alpha,j}\} = \{ j \geq 1 \mid M_{\alpha, j} \notin \HX[\beta] \}$ is infinite by \cref{item. limit ordinal}. So in both cases $\{j \geq 1 \mid T_j=M'_{\alpha,j}\}$ is infinite. Combining this with \cref{item. detail cohomology hierarchy}, we get that
\begin{equation}\label{eq. infinite dimension}
    H_n(M_\alpha,\{T_j\}_{j\geq 1}; \Z)\neq 0 \text{ for infinitely many } n.
\end{equation}
Consider the $M_\alpha$-module $\bigoplus_{j\geq 1} \Z[ M_\alpha/T_j]$. There is an augmentation $\bigoplus_{j\geq 1} \Z [M_\alpha/T_j]\to \Z$ sending each coset $gT_j$ to $1$. Let $\Delta$ be the kernel of this augmentation. According to \cite{bieri1978relative}, we have

    \begin{equation}\label{eq. 2}
        H_n(M_\alpha,\{T_j\}_{j\geq 1};\Z)=H_{n-1}(M_\alpha;\Delta).
    \end{equation}

As $X$ is contractible, \cite[VII Proposition 7.3]{brown1982cohomology} yields a natural isomorphism $H^{M_\alpha}_\ast(X;\Delta)\cong H_\ast(M_\alpha;\Delta)$. Let $X_p$ be the set of $p$-simplices of $X$ and let $\Sigma_p$ be a set of representatives of $M_\alpha\backslash X_p$. \cite[VII (7.7)]{brown1982cohomology}, combined with \eqref{eq. 2}, yields a spectral sequence
\[E^1_{p,q}=\bigoplus_{\sigma\in \Sigma_p} H_q(\mathrm{Stab}_{M_\alpha}(\sigma); \Delta) \Rightarrow H_{p+q+1}(M_\alpha,\{T_j\}_{j\geq 1};\Z).\]
Note that, in the notation of \cite[VII (7.7)]{brown1982cohomology}, $\Delta_\sigma = \Delta$, because $M_\alpha \curvearrowright X$ is admissible. 

We will now prove that $E^1_{p,q}=\{0\}$ whenever $q>1$. Combining this with $E^1_{p,q}=\{0\}$ whenever $p>\dim X$, we will get that $H_k(M_\alpha,\{T_j\}_{j\geq 1};\Z)=\{0\}$ whenever $k>\dim X+2$, which contradicts \eqref{eq. infinite dimension}. We need to prove that
    \begin{equation}\label{eq. vanishing hierarchy}
        H_q(\mathrm{Stab}_{M_\alpha}(\sigma); \Delta)=\{0\}\text{  for all $q>1$ and all }\sigma\in \Sigma_p.
    \end{equation}

Fix $q>1$ and $\sigma\in \Sigma_p$. If $\mathrm{Stab}_{M_\alpha}(\sigma)\cong \Z$, then \eqref{eq. vanishing hierarchy} clearly holds. Otherwise, by \cref{item. monster hierarchy}, the group $\mathrm{Stab}_{M_\alpha}(\sigma)$ must conjugate into some $M_{\alpha,j}$. By \cref{item. monster hierarchy} again, because $\mathrm{Stab}_{M_\alpha}(\sigma) \in \HX[\beta]$, the group $\mathrm{Stab}_{M_\alpha}(\sigma)$ must conjugate into some $T_j$ by definition of $T_j$. That is, there exists $g_0\in M_\alpha$ and $j_0\geq 1$ such that 

    \[g_0 \mathrm{Stab}_{M_\alpha}(\sigma) g^{-1}_0 \leq T_{j_0}.\]

Consider the short exact sequence

    \begin{equation*}
        0\to \Delta \to \bigoplus_{j\geq 1} \Z [M_\alpha/T_j] \to \Z \to 0.
    \end{equation*}
By \cite[III Proposition 6.1]{brown1982cohomology}, we have a long exact sequence
            \[\cdots\to H_q(\mathrm{Stab}_{M_\alpha}(\sigma);\Delta)\to H_q\left(\mathrm{Stab}_{M_\alpha}(\sigma);\bigoplus_{j\geq 1}\Z [M_\alpha/T_j]\right)\xrightarrow{\phi_q} H_q(\mathrm{Stab}_{M_\alpha}(\sigma); \Z)\to \cdots\]
For $j\neq j_0$, because the $\{T_j\}_{j \geq 1}$ form a malnormal family, we have that $\Z [M_\alpha/T_j]$ is a free $\mathrm{Stab}_{M_\alpha}(\sigma)$-module, with a free basis given by any set of double coset representatives of  
\[\mathrm{Stab}_{M_\alpha}(\sigma)\backslash M_\alpha/ T_j.\]
Let $R_0$ be a set of double coset representatives of
\[\mathrm{Stab}_{M_\alpha}(\sigma)\backslash M_\alpha/ T_{j_0}\]
such that $g_0\in R_0$. Because $T_{j_0}$ itself is malnormal, $\Z [M_\alpha/T_{j_0}]$ is a direct sum of a free $\mathrm{Stab}_{M_\alpha}(\sigma)$-module, generated by $tT_{j_0}: t\in R_0\smallsetminus\{g_0\}$, and a trivial $\mathrm{Stab}_{M_\alpha}(\sigma)$-module, generated by $g_0 T_{j_0}$.

This shows that
\[H_q\left(\mathrm{Stab}_{M_\alpha}(\sigma); \bigoplus_{j\geq 1} \Z[M_\alpha/T_j] \right) \cong H_q(\mathrm{Stab}_{M_\alpha}(\sigma); \Z),\]
so the map $\phi_q$ in the above long exact sequence is an isomorphism for all $q > 1$, and \eqref{eq. vanishing hierarchy} holds in this case as well.
\end{proof}

\begin{proof}[Proof of Theorem \ref{thm. hierarchy}]
For all countable ordinals $\alpha \geq 1$, we construct finitely generated torsion-free simple groups in $\HX[\alpha+1] \minus \HX[\alpha]$, by transfinite induction. By assumption, there exists a countable torsion-free group $M_0\in \HX[1]\minus \X$. For the inductive step, assuming that, for some countable ordinal $\alpha\geq 1$ and each ordinal $\beta<\alpha$, we have built a countable torsion-free group $M_\beta \in \HX[\beta+1] \minus \HX[\beta]$, we will then build a finitely generated torsion-free simple group $M_\alpha$ that satisfies \ref{item. cohomology hierarchy}. By \Cref{lem cohomology implies hierarchy}, this group must also lie in $\HX[\alpha+1] \minus \HX[\alpha]$, completing the induction step.

Suppose first that $\alpha$ is a countable limit ordinal. Enumerate the set of ordinals $\beta < \alpha$ as $\beta_1,\beta_2,\dots$ Set $M'_{\alpha,j}=M_{\beta_j}$, given by the induction hypothesis. Apply Theorem \ref{thm. embedding} with $H_1 = M'_{\alpha, j}$ and $L_1 = G_j$ -- and $H_i = \{ 1 \}, L_i = F_2: i \geq 2$ -- to build $M_{\alpha, j}$. Apply Theorem \ref{thm. embedding} to $H_i = M_{\alpha, i}, L_i = F_2 : i \geq 1$ to obtain $M_\alpha$. Then all items follow from Theorem \ref{thm. embedding} and the induction hypothesis.

Now suppose that $\alpha$ is a successor ordinal, and let $M_{\alpha-1} \in \HX[\alpha] \minus \HX[\alpha-1]$ be a countable torsion-free group, given by the induction hypothesis. For each $j \geq 1$, apply Theorem \ref{thm. embedding} with $H_1 = M_{\alpha-1}, L_1 = G_j$ -- and $H_i = \{ 1 \}, L_i = F_2 : i \geq 2$ -- to construct a group $N_{\alpha,j}$. Because $M_{\alpha-1} \in \HX[\alpha] \minus \HX[\alpha-1]$, we have $N_{\alpha, j} \in \HX[\alpha+1]$ by Corollary \ref{cor monsters HX}, and $N_{\alpha, j} \notin \HX[\alpha-1]$. For clarity, denote the copy of $M_{\alpha-1}$ in $N_{\alpha,j}$ as $N'_{\alpha,j}$.

Suppose first that, for some $j_0$, it holds that $N_{\alpha,j_0}\in \HX[\alpha+1]\minus \HX[\alpha]$. Then we will construct $M_\alpha$ as follows. For $j\neq j_0$, let $M'_{\alpha,j}=\{1\}$ and apply Theorem \ref{thm. embedding} with $H_1 = M'_{\alpha, j} = \{ 1 \}$ and $L_1 = G_j$ -- and $H_i = \{ 1 \}, L_i = F_2 : i \geq 2$ -- to construct a group $M_{\alpha,j}$. Let also $M_{\alpha,j_0}=N_{\alpha,j_0}, M'_{\alpha,j_0}=N'_{\alpha,j_0}$. Apply Theorem \ref{thm. embedding}; with $H_i=M_{\alpha,i},L_i = F_2 : i \geq 1$ to obtain a group $M_\alpha$. Then it follows directly from Theorem \ref{thm. embedding} that $M_\alpha$ is a finitely generated torsion-free simple group satisfying all items of \ref{item. cohomology hierarchy}, except for \ref{item. non-limit ordinal}. For $j\neq j_0$, Theorem \ref{thm. embedding} yields that $\cd(M_{\alpha,j})<\infty$, and so $M_{\alpha,j}\in \HX[1]$, which is \ref{item. non-limit ordinal}.

Suppose that $N_{\alpha,j}\in \HX[\alpha]$ for every $j$. Then we let $M_{\alpha,j}=N_{\alpha,j}$ and $M'_{\alpha,j}=N'_{\alpha,j}$. Apply Theorem \ref{thm. embedding}; with $H_i=M_{\alpha,i},L_i=F_2 : i \geq 1$ to obtain a group $M_\alpha$. Then it follows directly from Theorem \ref{thm. embedding} that $M_\alpha$ is a finitely generated torsion-free simple group satisfying \ref{item. cohomology hierarchy}.
\end{proof}

\subsection{Further constructions}
\label{ss. further}

We end by proving Theorems \ref{intro:thm:verballycomplete}, \ref{intro:thm:conjugacyclass} and \ref{intro:thm:characteristicquotients}. The constructions are completely analogous to that of \cite{hull2013small, characteristicquotients}, except that Hull's small cancellation theorem \cite[Theorem 7.1]{hull2013small} is replaced by our Theorem \ref{thm. inductive step} to gain additional homological control.

\begin{theorem}[Theorem \ref{intro:thm:verballycomplete}]
\label{thm:verballycomplete}
    Let $n \geq 2$. Let $G$ be a countable acylindrically hyperbolic group with $\gd(G) \leq n$. Then there exists a finitely generated quotient $M$ of $G$ that is verbally complete and such that $\hd_R(M) = \cd_R(M) = \gd(M) = n$.
\end{theorem}

\begin{proof}
    Let $L$ be the fundamental group of a closed orientable hyperbolic $n$-manifold. We apply Theorem \ref{thm. inductive step} to the free produt $G \ast L$, the suitable subgroup $G$ and a finite generating set of $L$, to obtain an acylindrically hyperbolic group $G(0)$ such that $H_n(L;R)$ embeds into $H_n(G(0); R)$ for all commutative unital rings $R$, and $\gd(G(0)) \leq n$, in particular $G(0)$ is torsion-free. Then we follow the proof of \cite[Theorem 7.8]{hull2013small}. This gives a sequence of torsion-free acylindrically hyperbolic groups $G(i)$ with quotients $G(i) \twoheadrightarrow G(i+1)$ factoring through an intermediate group $G(i+\frac{1}{2})$ such that the direct limit $M$ is finitely generated and verbally complete.

    The group $G(i+\frac{1}{2})$ is defined as a free product of $G(i)$ and a free group $J$ amalgamated over a cyclic group (the second case in the proof does not occur since $G(i)$ is torsion-free). So a $K(G(i + \frac{1}{2}), 1)$ can be obtained from any $K(G(i), 1)$ by wedging circles and attaching a single $2$-cell. Next, $G(i+1)$ is obtained by an application of \cite[Theorem 7.1]{hull2013small}, imposing that certain elements are absorbed by certain suitable subgroups. By applying instead Theorem \ref{thm. inductive step}, we can ensure that a $K(G(i+1), 1)$ can be obtained from a $K(G(i+\frac{1}{2}), 1)$ by attaching $2$-cells.

    This guarantees that a $K(M, 1)$ can be obtained from a $K(G(0), 1)$ wedged with infinitely many circles by attaching $2$-cells. This implies that $\gd(M) \leq n$, and by Lemma \ref{lem. compute cohomology} that $H_n(L; R) \neq 0$ embeds into $H_n(M;R)$, so that $\hd_R(M) \geq n$. So $\hd_R(M) = \cd_R(M) = \gd(M) = n$ and we conclude.
\end{proof}

\begin{theorem}[Theorem \ref{intro:thm:conjugacyclass}]
\label{thm:conjugacyclass}
    Let $n \geq 2$. Let $G$ be a countable acylindrically hyperbolic group with $\gd(G) \leq n$. Then there exists a quotient $M$ of $G$ that is finitely generated, has exactly two conjugacy classes, and such that $\hd_R(M) = \cd_R(M) = \gd(M) = n$.
\end{theorem}

\begin{proof}
    Once again, we start with an acylindrically hyperbolic quotient $G(0)$ of $G$ such that $H_n(L;R)$ embeds into $H_n(G(0);R)$ and $\gd(G(0)) \leq n$. Then we follow the proof of \cite[Theorem 7.9]{hull2013small}. This gives a sequence of torsion-free acylindrically hyperbolic groups $G(i)$ with quotients $G(i) \twoheadrightarrow G(i+1)$ factoring through an intermediate group $G(i+\frac{1}{2})$ such that the direct limit $M$ is finitely generated and has exactly two conjugacy classes.

    The group $G(i+\frac{1}{2})$ is either equal to $G(i)$, or defined as an HNN extension of $G_i$ along infinite cyclic subgroups. So, in the non-trivial case, a $K(G(i + \frac{1}{2}), 1)$ can be obtained from any $K(G(i), 1)$ by gluing a cylinder along its two boundary components. Next, $G(i+1)$ is obtained by an application of \cite[Theorem 7.1]{hull2013small}, imposing that certain elements are absorbed by certain suitable subgroups. By applying instead Theorem \ref{thm. inductive step}, we can ensure that a $K(G(i+1), 1)$ can be obtained from a $K(G(i+\frac{1}{2}), 1)$ by attaching $2$-cells.

    Once again, this shows that $\gd(M) \leq n$, and that $H_n(L;R)$ embeds into $H_n(M; R)$, so $n \leq \hd_R(M) \leq \cd_R(M) \leq \gd(M) \leq n$ and we conclude.
\end{proof}

\begin{theorem}[Theorem \ref{intro:thm:characteristicquotients}]
\label{thm:characteristicquotients}
    Let $d \geq 2(n-1)$. Then the free group $F_n$ admits an infinite simple characteristic quotient $M$ with $\hd_{\Q}(M) = \cd_{\Q}(M) = d$.
\end{theorem}

Following \cite{characteristicquotients}, this group will arise as a quotient of $\Aut(F_n)$. We focus on rational cohomological dimension because $\Aut(F_n)$ has torsion of all orders as $n$ grows. We single out the case of the free group since it is the most interesting one in view of Wiegold's problems, but the argument could be generalized further, as in \cite{characteristicquotients}.

\begin{proof}
    We start by showing that $\cd_{\Q}(\Aut(F_n)) = 2(n-1)$. Indeed, $\Out(F_n)$ has a finite-index subgroup $H$ with $\cd_{\Z}(H) = 2n-3$ \cite{outerspace}. Moreover, $\Out(F_n)$ contains a copy of $\Z^{2n-3}$, so $\cd_{\Q}(H) = 2n-3$ as well, and thus $\cd_{\Q}(\Out(F_n)) = 2n-3$ \cite[V 5.3]{dicksdunwoody}. Since $F_n$ is of finite type, \cite[Theorem 5.5]{bieri1981homdim} implies that $\cd_{\Q}(\Aut(F_n)) = 1 + \cd_{\Q}(\Out(F_n)) = 2(n-1)$.
    
    By \cite{genevoishorbez}, the group $\Aut(F_n)$ is acylindrically hyperbolic; moreover it has no non-trivial finite normal subgroups. By \cite[Corollary 2.11, Lemma 5.4]{characteristicquotients}, there exists an acylindrically hyperbolic group $G_0$, that is a common quotient of $F_n$ and $\Aut(F_n)$, such that every quotient of $G_0$ is a characteristic quotient of $F_n$. This is proved by an application of \cite[Theorem 7.1]{hull2013small}; by applying instead Theorem \ref{thm. inductive step}, we can ensure that a  $K(G_0, 1)$ is obtained from a $K(\Aut(F_n, 1))$ by attaching $2$-cells. Let $L$ be the fundamental group of a closed orientable hyperbolic $d$-manifold. We apply Theorem \ref{thm. inductive step} to $G_0 \ast L$ and the suitable subgroup $G_0$ to obtain an acylindrically hyperbolic group $G_1$ that is a quotient of $G_0$ with $H_d(G_1; \Q) \neq 0$, and a $K(G_1, 1)$ is obtained from a $K(\Aut(F_n, 1)) \vee K(L, 1)$ by attaching $2$-cells.
    
    By \cite[Theorem 3.5]{characteristicquotients}, $G_1$ has an infinite simple quotient $M$. This is proved by constructing a sequence of acylindrically hyperbolic groups $G_k$, with maps $G_k \twoheadrightarrow G_{k+1}$ given by \cite[Theorem 7.1]{hull2013small}. By applying instead Theorem \ref{thm. inductive step}, we can ensure that a $K(M, 1)$ is obtained from a $K(G_1, 1)$ by attaching $2$-cells. Lemma \ref{lem. compute cohomology} then implies that the quotient $\Aut(F_n) \ast L \to M$ induces an isomorphism in homology in degrees at least $3$, and an injection in degree $2$. This shows that $H_d(M; \Q) \supset H_d(L; \Q) \neq 0$, and moreover $H^i(M; A) \cong H^i(\Aut(F_n) \ast L; A) = 0$ for all $i > d = \max\{d, 2(n-1)\}$ and all $\Q M$-modules $A$. So $d \leq \hd_{\Q}(M) \leq \cd_{\Q}(M) \leq d$.
\end{proof}

\footnotesize

\bibliographystyle{alpha}
\bibliography{ref}

\medskip

\normalsize

\noindent{\textsc{Department of Pure Mathematics and Mathematical Statistics, University of Cambridge, UK}}

\noindent{\textit{E-mail address:} \texttt{ff373@cam.ac.uk}}

\medskip

\noindent{\textsc{Department of Mathematics, Michigan State University, East Lansing, MI, 48824, USA}}

\noindent{\textit{E-mail address:} \texttt{sunbin1@msu.edu}}

\end{document}